\documentclass[11pt,a4paper]{article}
\usepackage{multicol}
\usepackage{geometry}
\usepackage{multirow}
\usepackage{graphicx}
\usepackage{amsmath}
\usepackage{amsthm}
\usepackage{array}
\usepackage{amssymb}
\usepackage{graphicx}
\usepackage{abstract}
\usepackage{caption}
\usepackage{epstopdf}
\usepackage[numbers,sort&compress]{natbib}
\usepackage{subfigure}
\usepackage{multirow}
\usepackage[colorlinks,linkcolor=blue,anchorcolor=red,citecolor=red]{hyperref}
\usepackage[capitalize]{cleveref}
\usepackage{color}
\usepackage{soul}
\usepackage{authblk}
\newtheorem{theorem}{Theorem}
\newtheorem{lem}{Lemma}[section]
\newtheorem{example}{Example}
\newtheorem{myDef}{Definition}

\newtheorem{cor}{Corollary}
\newtheorem{rem}{Remark}

\begin{document}
	\title{\bf The neural network models with delays for solving absolute value equations}
	\author[a]{Dongmei Yu\thanks{Supported partially by the Natural Science Foundation of China (12201275). Email: yudongmei1113@163.com.}}
    \author[a]{Gehao Zhang\thanks{Email: GeHao.Zh@outlook.com.}}
    \author[b]{Cairong Chen\thanks{Corresponding author. Supported partially by the Natural Science Foundation of Fujian Province (Grant No. 2020J01166). Email address: cairongchen@fjnu.edu.cn.}}
    \author[c]{Deren Han\thanks{Supported partially by the National Natural Science Foundation of China (12131004) and the Ministry of Science and Technology of China (2021YFA1003600). Email address: handr@buaa.edu.cn.}}
   \affil[a]{Institute for Optimization and Decision Analytics, Liaoning Technical University, Fuxin, 123000, P.R. China.}
   \affil[b]{School of Mathematics and Statistics \& Key Laboratory of Analytical Mathematics and Applications (Ministry of Education) \& Fujian Provincial Key Laboratory of Statistics and Artificial Intelligence \& Fujian Key Laboratory of Analytical Mathematics and Applications (FJKLAMA) \& Center for Applied Mathematics of Fujian Province (FJNU), Fujian Normal University, Fuzhou, 350117, P.R. China.}
   \affil[c]{LMIB of the Ministry of Education, School of Mathematical Sciences, Beihang University, Beijing, 100191, P.R. China.}
    \maketitle

\begin{quote}
{\bf Abstract:} An inverse-free neural network model with mixed delays is proposed for solving the absolute value equation (AVE) $Ax -|x| - b =0$, which includes an inverse-free neural network model with discrete delay as a special case. By using the Lyapunov-Krasovskii theory and the  linear matrix inequality (LMI) method, the developed neural network models are proved to be exponentially convergent to the solution of the AVE. Compared with the existing neural network models for solving the AVE, the proposed models feature the ability of solving a class of AVE with $\|A^{-1}\|>1$. Numerical simulations are given to show the effectiveness of the two delayed neural network models.	

{\small \medskip {\bf Keywords:} Delayed neural network model; Absolute value equation; Lyapunov-Krasovskii theory; Linear matrix inequality; Exponential convergence.}

{\small \medskip {\bf Mathematics Subject Classification:} 65F10, 65H10, 90C30}
\end{quote}

\section{Introduction}\label{sec:intro}
We are interested in solving the absolute value equation (AVE) problem of finding an $x\in \mathbb{R}^n$ such that
\begin{equation}\label{eq:ave}
Ax - |x| - b = 0,
\end{equation}
where $A \in \mathbb{R}^{n \times n}$, $b \in \mathbb{R}^{n}$ and $|x|\in \mathbb{R}^n$  denotes the componentwise absolute value of $x$ whose the $i$-th component is $x_i$ if $x_i\ge 0$ and $-x_i$ otherwise. In fact, if $b=0$, then $x=0$ is clearly a solution of the AVE~\eqref{eq:ave}. In the rest of this paper, we will assume that $b\neq 0$ and the AVE~\eqref{eq:ave} is solvable.

Over the past two decades, the AVE~\eqref{eq:ave} has attracted more and more attention because of its relevance to many mathematical programming problems, such as the linear complementarity problem (LCP), the bimatrix game and so on; see, e.g., \cite{mang2006,mang2007,prok2009} and references therein. In general, solving the AVE~\eqref{eq:ave} is NP-hard \cite{mang2006} and it has been proved that if the AVE~\eqref{eq:ave} is solvable, it can have either a unique solution or multiple solutions \cite{hlad2023,wugu2016}.

For the solvable AVE~\eqref{eq:ave}, numerically, there are already various algorithms for solving it and most of which are available when it has a unique solution. Moreover, the existing numerical methods can roughly be viewed differently from a discrete perspective and a continuous perspective. Our focus in this paper is the latter. Concretely, from a continuous perspective, several neural network models (also known as dynamic models in the literature) have been investigated for solving the
AVE~\eqref{eq:ave}; see, e.g., \cite{gawa2014,hucu2017,maee2017,maer2018,cyyh2021,jlhh2021,
lyyc2022,ycyh2023,jyfc2023,sanc2019}. Among them, theoretically, the neural network models proposed in \cite{sanc2019,lyyc2022,jyfc2023,jlhh2021} are only available for the case that $\|A^{-1}\|<1$ while the models proposed in \cite{cyyh2021,ycyh2023} are available for the case that $\|A^{-1}\|\le 1$. In addition, the models proposed in \cite{hucu2017,maee2017,maer2018,gawa2014} involves the computation of $(A-I)^{-1}$ or $A^{-1}$, which may increase the computation load, especially for the large scale AVE~\eqref{eq:ave}. Hence, the goal of this paper is to develop inverse-free neural network models for solving the AVE~\eqref{eq:ave} and, theoretically, the proposed models feature the ability of solving a class of AVE~\eqref{eq:ave} with $\|A^{-1}\|>1$.

In order to achieve the goal of this paper, inspired by the researches of \cite{licx2005,hhgw2014}, delays is firstly introduced into the neural network models for solving the AVE~\eqref{eq:ave}. Neural network models with discrete or distributed delay have been studied for solving some optimization problems, such as the nonsmooth constrained pseudoconvex optimization problem \cite{xcqw2020}, the distributed convex optimization problem \cite{jqxl2022}, the constrained convex optimization problem \cite{zhlh2022}, the quadratic programming problem \cite{yaca2006}, the linear variational inequality \cite{chht2009,hhgw2014}, the linear projection equation \cite{licx2005} and the cardinality-constrained portfolio selection problem \cite{lewa2022}. As is well known, delays inevitably occur during the signal communication among neurons \cite{bami1991,fepo2010}.

The main contributions of this paper can be summarized as follows:
\begin{itemize}
  \item [(i)] An inverse-free neural network model with mixed delays is firstly proposed for solving the AVE~\eqref{eq:ave}, which includes an inverse-free neural network model with discrete delay as a special case.

  \item [(ii)] The properties and exponential stabilities of the proposed models are studied in detail. Compared with the proof of \cite[Theorem~1]{hhgw2014}, the Lyapunov-Krasovskii functional in our proof of Theorem~1 is simpler and the derivative of the initial function $\phi$ is not required in our analysis.

  \item [(iii)] Numerical simulations are given to demonstrate the effectiveness of the proposed models.
\end{itemize}

The rest of this paper is organized as follows. In Section~\ref{sec:prel}, we present some classical definitions and preliminary results relevant to our later developments.
In Section~\ref{sec:nn}, the proposed neural network models are formal described and their  theoretical properties are explored. Numerical simulations are given in
Section~\ref{sec:exam}. Conclusions are made in Section~\ref{sec:conc}.

\textbf{Notation}.
$\mathbb{R}^{n\times n}$ denotes the set of all $n \times n$ real matrices and $\mathbb{R}^{n}= \mathbb{R}^{n\times 1}$. $\mathbb{R}$ denotes the set of real numbers. The transposition of a matrix or a vector is denoted by $ ^\top $. $I$ is the identity matrix with suitable dimention. The inner product of two vectors in $\mathbb{R}^n$ is defined as $\langle x, y\rangle\doteq x^\top y= \sum\limits_{i=1}^n x_i y_i$ and $\| x \|\doteq\sqrt{\langle x, x\rangle} $ denotes the Euclidean norm. For a given constant $\tau >0$, $\mathcal{C}_{\tau} = \mathcal{C} ([-\tau, 0], \mathbb{R}^{n})$ denotes the Banach space of continuous mappings from $[-\tau,0]$ into $\mathbb{R}^{n}$ equipped with the supremum norm $\|\phi\|_s = \sup\limits_{\theta\in [-\tau_,0]} \|\phi(\theta)\|$. For $\Pi\in \mathbb{R}^{n\times n}$,  $\Pi \prec 0$ means that $\Pi$ is a negative definite matrix and $\|\Pi\|$ denotes the spectral norm of $\Pi$ which is defined by $\| \Pi\|\doteq \max \left\{ \| \Pi x \| : x \in \mathbb{R}^{n}, \|x\|=1 \right\}$. For a symmetric matrix $\Pi\in \mathbb{R}^{n\times n}$, write $\lambda_{\min}(\Pi)$ and $\lambda_{\max}(\Pi)$ respectively the smallest and largest eigenvalues of $\Pi$. For a function $f$, $\dot{f}$ denotes the derivative of $f$.

\section{Preliminaries}\label{sec:prel}
We begin with a class of retarded type time-delay dynamic models of the form \cite{guck2003}
\begin{equation}\label{ds}
\left\{
\begin{aligned}
&\frac{{\rm{d}} x(t)}{{\rm{d}} t}=f(t,x_{t}),\quad t\ge 0,\\
&x(t) = \phi(t), \quad  t \in \left[- h, 0\right],
\end{aligned}
\right.
\end{equation}
where $x_t\in \mathcal{C}_{h}$ is defined by $x_{t}(\xi)=x(t+\xi)$ with $\xi \in [-h,0]$\footnote{For any given $t\in \mathbb{R}$ and a continuous vector-valued function $x: \mathbb{R}\mapsto \mathbb{R}^n$, $x_t$ is the restriction of $x$ on the segment $[t-h,t]$.}, $\phi\in \mathcal{C}_{h}$ and $f: [0,\infty) \times \mathcal{C}_h \rightarrow \mathbb{R}^{n}$ is a functional. In the case where the initial function $\phi$ should be indicated explicitly we use the notation $x(t,\phi)$ to denote the solution of \eqref{ds}. Otherwise, $x(t)$ is used.

A vector $\hat{x}\in \mathbb{R}^n$ gives a state $\hat{x}_t\in \mathcal{C}_h$, which is a constant mapping on $[-h,0]$ with the constant value $\hat{x}\in \mathbb{R}^n$.

\begin{myDef}
A vector $x^{*} \in R^{n}$ is called an equilibrium point of \eqref{ds} with the initial state $x^*_0$ if $f(t,x_t^{*}) \equiv 0$ for $t\ge 0$.
\end{myDef}

The following definition is in the light of \cite[Definition~1.4.]{Kharitonov2012}.
\begin{myDef}\label{exponentially}
The equilibrium point $x^{*}$ of \eqref{ds} is said to be exponentially stable, if there exist $\Delta >0$, $\sigma >0$, and $\gamma >0$ such that for any initial function $\phi$ with $\|\phi(t) - x^*\|_s<\Delta$, the following inequality holds:
\begin{equation*}
	\| x(t) - x^{*} \| \leq \gamma \|\phi(t)-x^*\|_s {\rm e}^{- kt},\quad \forall t\ge 0.
\end{equation*}
\end{myDef}

In the following, we recall the definition of Euclidean projection and present its basic properties. Let $C$ be a nonempty closed convex subset of $\mathbb{R}^n$. Then there exists a unique vector $\bar{x}\in C$ that is closest to $x\in \mathbb{R}^n$ in the sense of Euclidean norm. The vector $\bar{x}$ is called the Euclidean projection of $x$ onto $C$ and denoted by $P_C[x]$. The mapping $P_C: x\mapsto P_C[x]$ is called the Euclidean projector onto $C$. For the Euclidean projector, we have the following properties.

\begin{lem}\label{ty}(\cite[1.5.5~Theorem]{fapa2003})
Let $C$ be a nonempty closed convex subset of $\mathbb{R}^n$. Then the following statements are valid.
\begin{itemize}
  \item [(i)] For any two vectors $x$ and $y$ in $\mathbb{R}^n$,
  $$\left\|P_{C}[x] - P_{C}[y]\right\| \leq \|x-y\|.$$

  \item [(ii)] For each $x\in \mathbb{R}^n$, $P_C[x]$ is the unique vector $\bar{x}\in C$ satisfying the inequality:
      $$(y - \bar{x})^\top (\bar{x}-x)\ge 0, \quad \forall y\in C.$$

  \item [(iii)] For any two vectors $x$ and $y$ in $\mathbb{R}^n$,
  $$ \left\|P_{C}[x] - P_{C}[y]\right\|^2  \leq \left(P_{C}[x] - P_{C}[y]\right)^{\top} (x - y).$$
\end{itemize}
\end{lem}

If $C = \{x\in \mathbb{R}^n: x\ge0\}$, the result of Lemma~\ref{ty} (iii) can be generalized as follows.
\begin{lem}\label{pro}
Let $D={\rm diag}(d_{i}) \in \mathbb{R}^{n \times n}$ be a positive definite diagonal matrix. Assume that $C = \{x\in \mathbb{R}^n: x\ge0\}= \bar{C}^n$ with $\bar{C}=\{x\in \mathbb{R}: x\ge0\}$. Then for any $x, y \in \mathbb{R}^{n}$, we have
\begin{equation*}
	(P_{C}[x] - P_{C}[y])^{\top}D(P_{C}[x] - P_{C}[y]) \leq \left[P_{C}(x) - P_{C}(y)\right]^{\top} D(x - y).
\end{equation*}
\end{lem}
\begin{proof}
Let $x$ and $y$ be two arbitrary vectors in $\mathbb{R}^n$. By Lemma~\ref{ty}~(ii), we have
\begin{equation}\label{ie:pr1}
(P_{\bar{C}}[x_i] - P_{\bar{C}}[y_i]) (P_{\bar{C}}[y_i] - y_i) \ge 0,\quad i=1,2,\cdots, n
\end{equation}
and
\begin{equation}\label{ie:pr2}
(P_{\bar{C}}[y_i] - P_{\bar{C}}[x_i]) (P_{\bar{C}}[x_i] - x_i) \ge 0,\quad i=1,2,\cdots, n.
\end{equation}
It follows respectively from \eqref{ie:pr1} and \eqref{ie:pr2} that
\begin{equation}\label{ie:pr1d}
(P_{\bar{C}}[x_i] - P_{\bar{C}}[y_i])d_i (P_{\bar{C}}[y_i] - y_i) \ge 0,\quad i=1,2,\cdots, n
\end{equation}
and
\begin{equation}\label{ie:pr2d}
(P_{\bar{C}}[y_i] - P_{\bar{C}}[x_i])d_i (P_{\bar{C}}[x_i] - x_i) \ge 0,\quad i=1,2,\cdots, n
\end{equation}
with $d_i>0\,(i=1,2,\cdots, n)$. Adding \eqref{ie:pr1d} and \eqref{ie:pr2d} and rearranging terms, we immediately obtain
\begin{equation}\label{ie:con1}
(P_{\bar{C}}[x_i] - P_{\bar{C}}[y_i])d_i (P_{\bar{C}}[x_i] - P_{\bar{C}}[y_i]) \le (P_{\bar{C}}[x_i] - P_{\bar{C}}[y_i])d_i(x_i-y_i),\quad i=1,2,\cdots, n.
\end{equation}
Then the result follows from \eqref{ie:con1} and $P_C[x] = (P_{\bar{C}}[x_1],P_{\bar{C}}[x_2],\cdots, P_{\bar{C}}[x_n])^\top$.
\end{proof}

Before ending this section, we introduce the following integral inequality which plays a key role for the result of the next section.

\begin{lem}\label{Gu}(\cite[Lemma~1]{fela2011})
For any symmetric positive definite matrix $\mathcal{K} \in \mathbb{R}^{n \times n}$ and vector-valued function $\omega: [a,b] \rightarrow \mathbb{R}^{n}$ with $b > a$ such that the following integrations are well defined, then
\begin{equation*}
(b-a) \left[\int_{a}^{b} \left(\omega(s)\right)^{\top}\mathcal{K}\omega(s)ds\right] \geq \left[\int_{a}^{b}\omega(s)ds\right]^{\top}\mathcal{K}\left[\int_{a}^{b}\omega(s)ds\right].
\end{equation*}
\end{lem}

\section{The neural network model with mixed delays}\label{sec:nn}
In this section, we will develop a neural network model with mixed delays for solving the AVE~\eqref{eq:ave} and then carry out its properties.

In \cite{mang2006}, Mangasarian shows that the AVE~\eqref{eq:ave} is equivalent with the generalized
LCP~(GLCP) of finding an $x \in \mathbb{R}^{n}$ such that
\begin{equation}\label{eq:glcp}
Q(x) \dot{=} Ax + x - b \geq 0, \quad F(x) \dot{=} Ax-x-b \geq 0 \quad \text{and} \quad  \left\langle Q(x),F(x) \right\rangle =0.
\end{equation}
In addition, GLCP \eqref{eq:glcp} can be equivalently reformulated as the following extended linear  variational inequality (ELVI) problem \cite{gao2001}:
\begin{equation}\label{eq:gvi}
\text{find an}~x^*\in \mathbb{R}^n~\text{such that}~Q(x^*)\in \Omega, ~ \left\langle v-Q(x^*),F(x^*)\right\rangle\ge 0,\; \forall\, v\in \Omega,
\end{equation}
where $\Omega = \{x\in \mathbb{R}^n: x\ge0\}$. Moreover, $x^*$ is a solution of ELVI \eqref{eq:gvi} if and only if it satisfies the following generalized linear projection equation
\begin{equation*}
Q(x) = P_{\Omega}\left[Q(x)-F(x)\right].
\end{equation*}
Based on this observation, the projection neural network with mixed delays for solving the linear variational inequality \cite{hhgw2014} is extended to solve the ELVI \eqref{eq:gvi} or the AVE~\eqref{eq:ave}. Specifically, we propose the following neural network model with mixed delays  (NNMMD) for solving the AVE~\eqref{eq:ave}:
\begin{equation}\label{eq:mtdnn}\footnotesize
\begin{cases}
\frac{{\rm{d}}x(t)}{{\rm{d}}t}= g(t,x_t)\doteq -[A + (2+\tau_2)I] x_t(0) + P_{\Omega}[2x_t(0)] + x_t(-\tau_1) + \int_{-\tau_2}^0x_t(s){\rm d}s + b,\quad t \in [0, +\infty),\\
x(t) =\phi(t),\quad t \in \left[- \tau_M, 0\right],
\end{cases}
\end{equation}
where $\tau_1,\tau_2 \ge 0$, $\tau_M=\max\{\tau_1,\tau_2\}$, and $\phi(t) \in \mathcal{C}_{\tau_M}$. Since $x_{t}(\xi)= x(t+\xi)$, it is clear that
 \begin{equation}\label{eq:gtt}
 g(t,x_t) = -[A + (2+\tau_2)I] x(t) + P_{\Omega}[2x(t)] + x(t-\tau_1) + \int_{t-\tau_2}^tx(s){\rm d}s + b.
 \end{equation}
Moreover, it can be concluded that a constant vector $x^*$ is a solution of the AVE~\eqref{eq:ave} if and only if it is an equilibrium point of the NNMMD~\eqref{eq:mtdnn}.

In the following, we will study the existence and uniqueness of the solution to the NNMMD~\eqref{eq:mtdnn}.

\begin{lem}\label{solution}
For each $\phi(t) \in \mathcal{C}_{\tau_M}$, there exists a unique continuous solution to the  NNMMD~\eqref{eq:mtdnn} in the global time interval $[0, +\infty)$.
\end{lem}
\begin{proof}
For any $t\in [0,\infty)$ and $\psi_t \in \mathcal{C}_{\tau_M}$ with $\|\psi_t\|_s\le H$, we have
\begin{align}\nonumber
\| g(t,\psi_t)\| \leq &\left\| [A + (2+\tau_2)I]\psi_{t}(0)\right\| + 2\|\psi_{t}(0)\| + \|\psi_{t}(-\tau_1)\|
+\int_{-\tau_{2}}^{0}\left\|\psi_{t}(s)\right \| {\rm d}s +\|b\| \\\label{ie:psi}
\leq &(\| A \| + 5 +2\tau_{2}) \|\psi_t \|_{h}+\|b\|\\\nonumber
\leq &(\| A \| + 5 +2\tau_{2}) H +\|b\|.
\end{align}
We should mention that the first inequality use the fact that $\|P_{\Omega}[x]\|\le \|x\|$. It follows from \eqref{ie:psi} that
$$
\| g(t,\psi_t)\| \leq \eta(\|\psi_t\|_h)
$$
for $t\ge 0$ and $\psi_t \in \mathcal{C}_{\tau_M}$, where $\eta(r) = (\| A \| + 5 +2\tau_{2}) r +\|b\|$ with $r\in [0,\infty)$ is continuous and increasing, and for any $r_{0} \geq 0$, we have
\begin{equation*}
\lim\limits_{R \rightarrow \infty} \int_{r_{0}}^{R}\frac{{\rm d}r}{\eta(r)}=\lim\limits_{R \rightarrow \infty} \int_{r_{0}}^{R}\frac{{\rm d}r}{(\|A\|+5+2\tau_{2})r+\|b\|}=\infty.
\end{equation*}

In the following, we will show that the functional $g$ is Lipschitz continuous with respect to the second argument. To this end, for any $t\ge 0$ and $\phi_t^{1}, \phi_t^{2} \in \mathcal{C}_{\tau_M}$, it follows from the triangle inequality and Lemma~\ref{ty}~(i) that
\begin{align*}
\left\| g(t,\phi_t^{1}) - g(t, \phi_t^{2}) \right\| & \leq \left\| [A + (2+\tau_2)I][\phi_t^{1}(0) - \phi_t^{2}(0)] \right\| + 2\| \phi_t^{1}(0) - \phi_t^{2}(0) \|  \\
&\quad + \|\phi_t^{1}(-\tau_1) - \phi_t^{2}(-\tau_1)\|
    +\int_{-\tau_{2}}^{0}\left\|\phi_t^{1}(s)-\phi_t^{2}(s) \right \| {\rm d}s \\
	&\leq (\| A \| + 5 +2\tau_{2}) \|\phi_t^{1} - \phi_t^{2} \|_{s},
	\end{align*}
which implies that the functional $g$ is Lipschitz continuous with respect to the second argument with Lipschitz constant $\| A \| + 5 +2\tau_{2}$. This fact also means that the functional $g$ is continuous with respect to the second argument.

Finally, according to \cite[Theorem~1.1. and Theorem~1.2.]{Kharitonov2012}, the proof is completed by showing that the functional $g$ is continuous with respect to the first argument. Without loss of generality, let $0\le t_1 < t_2$ and $x_{t_1}, x_{t_2} \in \mathcal{C}_{\tau_M}$ with $x$ being a continuous function in~$\mathbb{R}$. Then $x_{t_2}$ is uniformly continuous in $[-\tau_M,0]$, and the uniformly continuous interval can be extended to $[-(\tau_M+\delta),0]$ for some $\delta>0$ since $x$ is continuous in $\mathbb{R}$. Hence, for any $\bar{\epsilon}>0$, there is a $\bar{\delta}>0$ such that for any $\theta_1,\theta_2\in [-(\tau_M+\delta),0]$ with $|\theta_1-\theta_2|<\bar{\delta}$, we have $\|x_{t_2}(\theta_1)-x_{t_2}(\theta_2)\|<\bar{\epsilon}$. Then for any $\epsilon = (\|A\|+5+2\tau_{2}) \bar{\epsilon} > 0$, there is a $0< \delta\le \bar{\delta}$ such that $|t_1- t_2| < \delta$ implies
\begin{align*}
 \|g(t_{1},x_{t_{1}})- g(t_{2},x_{t_{2}})\| &= \Big\| [A + (2+\tau_2)I] \left[x_{t_{2}}(0)-x_{t_{1}}(0)\right] + P_{\Omega}[2x_{t_{1}}(0)]-P_{\Omega}[2x_{t_{2}}(0)]\\
 &\quad + x_{t_{1}}(-\tau_1)-x_{t_{2}}(-\tau_1) + \int_{-\tau_2}^0 \left[x_{t_{1}}(s)-x_{t_{2}}(s)\right]{\rm d}s \Big\| \\
 &=\Big\| [A + (2+\tau_2)I] \left[x_{t_{2}}(0)-x_{t_{2}}(t_1 - t_2)\right] + P_{\Omega}[2x_{t_{2}}(t_1-t_2)]-P_{\Omega}[2x_{t_{2}}(0)]\\
 &\quad + x_{t_{2}}(t_1 - t_2-\tau_1)-x_{t_{2}}(-\tau_1) + \int_{-\tau_2}^0 \left[x_{t_{2}}( t_1 - t_2 + s)-x_{t_{2}}(s)\right]{\rm d}s \Big\| \\
 &< (\|A\|+5+2\tau_{2})\bar{\epsilon}=\epsilon.
\end{align*}
Hence, the functional $g$ is continuous with respect to $t$.
\end{proof}	
		
Now we are in the position to consider the stability of the equilibrium point of the NNMMD~\eqref{eq:mtdnn}.

\begin{theorem}\label{stable:mtdnn}
If there exist three symmetric positive-definite matrices $P \in \mathbb{R}^{n \times n }$, $Q \in \mathbb{R}^{n \times n}$ and $H \in \mathbb{R}^{n \times n}$, and a positive-definite diagonal matrix $D={\rm diag}(d_{i}) \in \mathbb{R}^{n \times n}$, and two matrices $T_{1}$ and $T_{2}\in \mathbb{R}^{n \times n}$ as well as a constant $k>0$ such that
	\begin{equation}\label{2LMI}
	\Pi=
	\left[
	\begin{array}{cccccc}
	\Pi_{11} & \Pi_{12} & \Pi_{13} & \Pi_{14}  & P \\
	\circ & \Pi_{22} & \Pi_{23} & T_{1}^{\top}  & 0 \\
	\circ & \circ & \Pi_{33} & \Pi_{34}  & \Pi_{35} \\
	\circ & \circ & \circ & \Pi_{44}  & \Pi_{45} \\
	\circ & \circ & \circ & \circ  & \Pi_{55} \\
	\end{array}
	\right] \prec 0,
	\end{equation}
then the equilibrium point of the NNMMD~\eqref{eq:mtdnn} is exponentially stable. In \eqref{2LMI}, $\circ$ denotes entries which can be readily inferred by symmetry and
\begin{equation*}
\begin{array}{ll}
\Pi_{11}=(2k-4-2 \tau_{2} )P-A^{\top}P-PA + Q + \tau_{2} H , &\Pi_{12}= P, \\
\Pi_{13}=P+2kD-A^{\top}D-\tau_{2} D -(A+2I+\tau_{2}I)^{\top}T_{2}^{\top},&
\Pi_{14}= -(A+2I+\tau_{2}I)^{\top}T_{1}^{\top},\\
\Pi_{22}=-{\rm e}^{-2k\tau_{1}}Q,& \Pi_{23}=D+T_{2}^{\top},    \\
\Pi_{33}= -2D+T_{2}+T_{2}^{\top}, &\Pi_{34}= T_{1}^{\top}-T_{2},  \\
\Pi_{35}= T_{2}+D,  &\Pi_{44}= -T_{1}-T_{1}^{\top},    \\
\Pi_{45}= T_{1},  &\Pi_{55}=-\frac{{\rm e}^{-2k\tau_{2}}}{\tau_{2}}H.
\end{array}
\end{equation*}
\end{theorem}

\begin{proof}
Let $x^{*}\in \mathbb{R}^n$ be an equilibrium point of the NNMMD~\eqref{eq:mtdnn}, then we have
\begin{equation}\label{eq:equi}
	-(A+2I)x^{*} + P_{\Omega}[2x^{*}] + x^{*} + b + \int_{t - \tau}^{t}x^{*} {\rm d}s - \tau x^{*}=0.
\end{equation}
Constructing the following Lyapunov-Krasovskii functional
\begin{align*}
V(x(t)) &= {\rm e}^{2kt} (x(t) - x^{*})^{\top} P (x(t) - x^{*}) + 2{\rm e}^{2kt} \sum\limits_{i=1}^{n}  d_{i} \int_{x_{i}^{*}}^{x_{i}(t)} (P_{\Omega}[2z_{i}]-P_{\Omega}[2x_{i}^{*}]) {\rm d}z_{i} \\\nonumber
&\quad + \int_{t - \tau_{1}}^{t} {\rm e}^{2ks} (x(s) - x^{*})^{\top} Q (x(s) - x^{*}){\rm d}s +\int_{t-\tau_{2}}^{t}\int_{\theta}^{t}{\rm e}^{2ks}(x(s) - x^{*})^{\top} H (x(s) - x^{*}){\rm d}s {\rm d} \theta.
\end{align*}
According to the definition of $P_{\Omega}$, we know that $P_{\Omega}$ is a nondecreasing function, from which we can prove that
\begin{align}\label{ineq}
	0 \leq \int_{x_{i}^{*}}^{x_{i}(t)} (P_{\Omega}[2z_{i}] - P_{\Omega}[2x_{i}^{*}]) {\rm d} z_{i}\leq (x_{i}(t) - x_{i}^{*}) (P_{\Omega}[2x_{i}(t)] - P_{\Omega}[2x_{i}^{*}]).
\end{align}
Indeed, on the one hand, if $x_{i}(t) \geq x_{i}^{*}$, then \eqref{ineq} holds. On the other hand, if $x_{i}^{*} \geq x_{i}(t)$, then we get
\begin{equation*}
	0 \geq \int_{x_{i}(t)}^{x_{i}^{*}} (P_{\Omega}[2z_{i}] - P_{\Omega}[2x_{i}^{*}]) {\rm d}z_{i} \geq -(x_{i}(t) - x_{i}^{*}) (P_{\Omega}[2x_{i}(t)] - P_{\Omega}[2x_{i}^{*}]),
\end{equation*}
which also implies \eqref{ineq}. It follows from \eqref{ineq} and the positive definiteness of the symmetric matrices $P,Q$ and $H$ that $V(x(t))\ge 0$ and $V(x(t)) = 0$ if and only if $x(t)\equiv x^*$.

In the following, we will show the derivative of the Lyapunov-Krasovskii functional $V$ along the solution of the NNMMD~\eqref{eq:mtdnn}. Firstly, we have
\begin{align*}
\frac{{\rm {d}}V(x(t))}{{\rm {d}}t} &=  2k{\rm e}^{2kt} (x(t) - x^{*})^{\top} P (x(t) - x^{*}) + 2{\rm e}^{2kt}\left(\frac{{\rm {d}}x(t)}{{\rm {d}}t}\right)^{\top} P (x(t) - x^{*}) \\
&\quad + 4k{\rm e}^{2kt} \sum\limits_{i=1}^{n} d_{i} \int_{x_{i}^{*}}^{x_{i}(t)} (P_{\Omega}[2z_{i}] - P_{\Omega}[2x_{i}^{*}]) {\rm d}z_{i} + 2{\rm e}^{2kt}\left(\frac{{\rm {d}}x(t)}{{\rm {d}}t}\right)^{\top} D (P_{\Omega}[2x(t)] - P_{\Omega}[2x^{*}])\\
&\quad + {\rm e}^{2kt} (x(t) - x^{*})^{\top} Q (x(t) - x^{*}) - {\rm e}^{2k(t - \tau_{1})} (x(t - \tau_{1}) - x^{*})^{\top} Q (x(t - \tau_{1}) - x^{*})\\
&\quad + \tau_{2} {\rm e}^{2kt}(x(t)-x^{*})^{\top}H(x(t)-x^{*})-\int_{t-\tau_{2}}^{t}{\rm e}^{2ks}(x(s)
-x^{*})^{\top}H(x(s)-x^{*}){\rm d}s,
\end{align*}
from which and \eqref{eq:gtt} we have
\begin{align}\nonumber
\frac{{\rm {d}}V(x(t))}{\rm {d}t}&=  2k{\rm e}^{2kt}(x(t) -x ^{*})^{\top} P (x(t) -x^{*})
+2{\rm e}^{2kt}{\Big[}(-A+2I)x(t) + P_{\Omega}[2x(t)] + b + x(t - \tau_{1})\\\nonumber
&\quad +\int_{t-\tau_{2}}^{t}x(s){\rm d}s-\tau_{2} x(t){\Big]}^{\top} P (x(t) - x^{*})\\\nonumber
&	 +   4k{\rm e}^{2kt} \sum\limits_{i=1}^{n} d_{i} \int_{x_{i}^{*}}^{x_{i}(t)}(P_{\Omega}[2z_{i}] - P_{\Omega}[2x_{i}^{*}]) {\rm d}z_{i}
+2{\rm e}^{2kt}{\Big[}(-A+2I)x(t) + P_{\Omega}[2x(t)] + b + x(t - \tau_{1})\\\nonumber
&\quad +\int_{t-\tau_{2}}^{t}x(s){\rm d}s-\tau_{2} x(t){\Big]}^{\top}D (P_{\Omega}[2x(t)] - P_{\Omega}[2x^{*}])\\\nonumber
	&\quad +  {\rm e}^{2kt} (x(t)- x^{*})^{\top} Q (x(t) - x^{*})- {\rm e}^{2k(t - \tau_{1})} (x(t - \tau_{1}) - x^{*})^{\top} Q (x(t - \tau_{1}) - x^{*})\\
	&\quad +\tau_{2} {\rm e}^{2kt}(x(t)-x^{*})^{\top}H(x(t)-x^{*})-\int_{t-\tau_{2}}^{t}{\rm e}^{2ks}(x(s)-x^{*})^{\top}H(x(s)-x^{*}){\rm d}s.\label{eq:demt}
\end{align}
Let $G(t) = P_{\Omega}[2x(t)] - P_{\Omega}[2x^{*}]$, then it can be obtained from \eqref{ineq} and \eqref{eq:demt} that
\begin{align}\nonumber
\frac{{\rm {d}}V(x(t))}{{\rm {d}}t}&\leq   2k{\rm e}^{2kt}(x(t) -x ^{*})^{\top} P (x(t) -x^{*}) +  4k{\rm e}^{2kt}(x(t)-x^{*})^{\top}DG(t)\\\nonumber
	&\quad +2{\rm e}^{2kt}\left[-(A+2I)x(t) + P_{\Omega}[2x(t)] + b + x(t - \tau_{1})+\int_{t-\tau_{2}}^{t}x(s){\rm d}s-\tau_{2} x(t)\right]^{\top} P (x(t) - x^{*})\\\nonumber
	&\quad +2{\rm e}^{2kt}\left[(-A+2I)x(t) + P_{\Omega}[2x(t)] + b + x(t - \tau_{1})+\int_{t-\tau_{2}}^{t}x(s){\rm d}s-\tau_{2} x(t)\right]^{\top}DG(t)\\\nonumber
	&\quad +  {\rm e}^{2kt} (x(t)- x^{*})^{\top} Q (x(t) - x^{*})- {\rm e}^{2k(t - \tau_{1})} (x(t - \tau_{1}) - x^{*})^{\top} Q (x(t - \tau_{1}) - x^{*})\\
	&\quad + \tau_{2} {\rm e}^{2kt}(x(t)-x^{*})^{\top}H(x(t)-x^{*})-{\rm e}^{2k(t-\tau_{2})}\int_{t-\tau_{2}}^{t}(x(s)-x^{*})^{T}H(x(s)-x^{*}){\rm d}s.\label{ie:de}
	\end{align}
Then it follows from \eqref{eq:equi} and \eqref{ie:de} that
\begin{align*}
	\frac{{\rm {d}}V(x(t))}{{\rm {d}}t} &\leq  2k{\rm e}^{2kt}(x(t)-x^{*})^{\top}P(x(t)-x^{*})+4k{\rm e}^{2kt}(x(t)-x^{*})^{\top}DG(t)\\
	&\quad -2{\rm e}^{2kt}(x(t)-x^{*})^{\top}(A+2I)^{\top}P(x(t)-x^{*})\\
	&\quad +2{\rm e}^{2kt}G^{\top}(t)P(x(t)-x^{*})+2{\rm e}^{2kt}(x(t-\tau_{1})-x^{*})^{\top}P(x(t)-x^{*})\\
	&\quad +2{\rm e}^{2kt}\left(\int_{t-\tau_{2}}^{t}(x(s)-x^{*}){\rm d}s\right)^{\top} P (x(t)-x^{*})-2{\rm e}^{2kt}\tau_{2} (x(t)-x^{*})^{\top}P(x(t)-x^{*})\\
	&\quad -2{\rm e}^{2kt}(x(t)-x^{*})^{\top}(A+2I)^{\top}DG(t)\\
	&+2{\rm e}^{2kt}G^{\top}(t)DG(t)+2{\rm e}^{2kt}(x(t-\tau_{1})-x^{*})^{\top}DG(t)\\
	&\quad +2{\rm e}^{2kt}\left(\int_{t-\tau_{2}}^{t}(x(s)-x^{*}){\rm d}s\right)^{\top} DG(t)-2{\rm e}^{2kt}\tau_{2} (x(t)-x^{*})^{\top}DG(t)\\
	&\quad  +  {\rm e}^{2kt} (x(t)- x^{*})^{\top} Q (x(t) - x^{*})- {\rm e}^{2k(t - \tau_{1})} (x(t - \tau_{1}) - x^{*})^{\top} Q (x(t - \tau_{1}) - x^{*})\\
	&\quad +\tau_{2} {\rm e}^{2kt}(x(t)-x^{*})^{\top}H(x(t)-x^{*})-{\rm e}^{2k(t-\tau_{2})}\int_{t-\tau_{2}}^{t}(x(s)-x^{*})^{\top}H(x(s)-x^{*}){\rm d}s.
\end{align*}
By employing Lemma \ref{Gu}, the following inequality holds
	\begin{equation*}
	-\int_{t-\tau_{2}}^{t} (x(s)-x^{*})^{\top}H(x(s)-x^{*}){\rm d}s \leq -\frac{1}{\tau_{2}}\left(\int_{t-\tau_{2}}^{t}(x(s)-x^{*}){\rm d}s\right)^{\top}H\int_{t-\tau_{2}}^{t}(x(s)-x^{*}){\rm d}s,
	\end{equation*}	
Therefore, we get
\begin{align*}
	\frac{{\rm {d}}V(x(t))}{{\rm {d}}t} &\leq  2k{\rm e}^{2kt}(x(t)-x^{*})^{\top}P(x(t)-x^{*})+4k{\rm e}^{2kt}(x(t)-x^{*})^{\top}DG(t)\\
	&\quad -2{\rm e}^{2kt}(x(t)-x^{*})^{\top}(A+2I)^{\top}P(x(t)-x^{*})\\
	&\quad +2{\rm e}^{2kt}G^{\top}(t)P(x(t)-x^{*})+2{\rm e}^{2kt}(x(t-\tau_{1})-x^{*})^{\top}P(x(t)-x^{*})\\
	&\quad +2{\rm e}^{2kt}\left(\int_{t-\tau_{2}}^{t}(x(s)-x^{*}){\rm d}s\right)^{\top} P (x(t)-x^{*})-2{\rm e}^{2kt}\tau_{2} (x(t)-x^{*})^{\top}P(x(t)-x^{*})\\
	&\quad -2{\rm e}^{2kt}(x(t)-x^{*})^{\top}(A+2I)^{\top}DG(t)\\
	&+2{\rm e}^{2kt}G^{\top}(t)DG(t)+2{\rm e}^{2kt}(x(t-\tau_{1})-x^{*})^{\top}DG(t)\\
	&\quad +2{\rm e}^{2kt}\left(\int_{t-\tau_{2}}^{t}(x(s)-x^{*}){\rm d}s\right)^{\top} DG(t)-2{\rm e}^{2kt}\tau_{2} (x(t)-x^{*})^{\top}DG(t)\\
	&\quad  +  {\rm e}^{2kt} (x(t)- x^{*})^{\top} Q (x(t) - x^{*})- {\rm e}^{2k(t - \tau_{1})} (x(t - \tau_{1}) - x^{*})^{\top} Q (x(t - \tau_{1}) - x^{*})\\
	&\quad +\tau_{2} {\rm e}^{2kt}(x(t)-x^{*})^{\top}H(x(t)-x^{*})-{\rm e}^{2k(t-\tau_{2})}\frac{1}{\tau_{2}}\left(\int_{t-\tau_{2}}^{t}(x(s)-x^{*}){\rm d}s\right)^{\top}H\int_{t-\tau_{2}}^{t}(x(s)-x^{*}){\rm d}s.
	\end{align*}
According to Lemma \ref{pro}, we have
\begin{align*}
\frac{{\rm {d}}V(x(t))}{{\rm {d}}t} &\leq  2k{\rm e}^{2kt}(x(t)-x^{*})^{\top}P(x(t)-x^{*})+4k{\rm e}^{2kt}(x(t)-x^{*})^{\top}DG(t)\\
&\quad -2{\rm e}^{2kt}(x(t)-x^{*})^{\top}(A+2I)^{\top}P(x(t)-x^{*})\\
&\quad +2{\rm e}^{2kt}G^{\top}(t)P(x(t)-x^{*})+2{\rm e}^{2kt}(x(t-\tau_{1})-x^{*})^{\top}P(x(t)-x^{*})\\
&\quad +2{\rm e}^{2kt}\left(\int_{t-\tau_{2}}^{t}(x(s)-x^{*}){\rm d}s\right)^{\top} P (x(t)-x^{*})-2{\rm e}^{2kt}\tau_{2} (x(t)-x^{*})^{\top}P(x(t)-x^{*})\\
	&\quad -2{\rm e}^{2kt}(x(t)-x^{*})^{\top}A^{\top}DG(t)\\
	&\quad -2{\rm e}^{2kt}G^{\top}(t)DG(t)+2{\rm e}^{2kt}(x(t-\tau_{1})-x^{*})^{\top}DG(t)\\
	&\quad +2{\rm e}^{2kt}\left(\int_{t-\tau_{2}}^{t}(x(s)-x^{*}){\rm d}s\right)^{\top} DG(t)-2{\rm e}^{2kt}\tau_{2} (x(t)-x^{*})^{\top}DG(t)\\
	&\quad  +  {\rm e}^{2kt} (x(t)- x^{*})^{\top} Q (x(t) - x^{*})- {\rm e}^{2k(t - \tau_{1})} (x(t - \tau_{1}) - x^{*})^{\top} Q (x(t - \tau_{1}) - x^{*})\\
	&\quad +\tau_{2} {\rm e}^{2kt}(x(t)-x^{*})^{\top}H(x(t)-x^{*})-{\rm e}^{2k(t-\tau_{2})}\frac{1}{\tau_{2}}\left(\int_{t-\tau_{2}}^{t}(x(s)-x^{*}){\rm d}s\right)^{\top}H\int_{t-\tau_{2}}^{t}(x(s)-x^{*}){\rm d}s.
\end{align*}
Moreover, for any real matrices $T_{1}$ and $T_{2}$ with compatible dimensions, we have
\begin{align}\nonumber
0&=2{\rm e}^{2kt}\left[\dot{x}(t)^{\top}T_{1}+G^{\top}(t)T_{2}\right][-\dot{x}(t)-(A+2I)(x(t)-x^{*})\\\label{lebnizi2}
	&\quad +G(t)+(x(t-\tau_{1})-x^{*})+\int_{t - \tau_{2}}^{t}(x(s)-x^{*}){\rm d}s-\tau_{2}(x(t)-x^{*})].
\end{align}
	Adding the terms on the right hand side of \eqref{lebnizi2} into $\frac{{\rm{d}}V(x(t))}{{\rm{d}}t}$ yields
	\begin{equation*}
	\frac{{\rm {d}}V(x(t))}{{\rm {d}}t} \leq {\rm e}^{2kt} v(t)^{T} \Pi v(t),
	\end{equation*}
	where
	\begin{equation*}
	v^{T}(t)=\left[(x(t)-x^{*})^{\top}   \quad   (x(t-\tau_{1})-x^{*})^{\top}  \quad G^{\top}(t)  \quad  \dot{x}(t)^{\top}   \quad  \left(\int_{t-\tau_{2}}^{t}(x(s)-x^{*}){\rm d}s\right)^{\top}\right]
	\end{equation*}
and $\Pi$ is shown in \eqref{2LMI}. If $\Pi \prec 0$ and $x(t) \not\equiv x^{*}$, we have $\frac{{\rm d}V(x(t))}{{\rm d}t} <0$. Therefore, we have $V(x(t)) \leq V(x(0))$. Furthermore, we have
\begin{align*}
	V(x(0)) &=  (x(0) - x^{*})^{\top} P (x(0) - x^{*})+ 2  \sum\limits_{i=1}^{n} d_{i} \int_{x_{i}^{*}}^{x_{i}(0)} (P_{\Omega}[2z_{i}] - P_{\Omega}[2x_{i}^{*}]) {\rm d}z_{i} \\
	&\quad + \int_{ - \tau_{1}}^{0} {\rm e}^{2ks} (x(s) - x^{*})^{\top} Q (x(s) - x^{*}){\rm d}s+\int_{-\tau_{2}}^{0}\int_{\theta}^{0}{\rm e}^{2ks}(x(s) - x^{*})^{\top} H (x(s) - x^{*}){\rm d}s{\rm d}\theta
\end{align*}
Consequently, we have
	\begin{align*}
	V(x(0)) &\leq\lambda_{\max}(P)\|x(0)-x^{*}\|^{2}+2~\max(d_{i})\|x(0)-x^{*}\|^{2}\\
	&\quad +\lambda_{\max}(Q)\int_{-\tau_{1}}^{0}e^{2ks}\|x(s)-x^{*}\|^{2}{\rm d}s+\lambda_{\max}(H)\int_{-\tau_{2}}^{0}\int_{\theta}^{0}{\rm e}^{2ks}\|(x(s) - x^{*})\|^{2}{\rm d}s{\rm d}\theta\\
	&\leq [\lambda_{\max}(P)+\frac{1-{\rm e}^{-2k\tau_{1}}}{2k}\lambda_{\max}(Q)+2\max(d_{i})+\frac{2k\tau_{2}-1+{\rm e}^{-2k\tau_{2}}}{4k^{2}}\lambda_{\max}(H)]\|\phi(t)-x^{*}\|_{s}^{2}.
	\end{align*}
Meanwhile, it is easy to see that
	\begin{align*}
	V(x(t))
	\geq & {\rm e}^{2kt}\lambda_{\min}(P)\|x(t) - x^{*}\|^{2}.
	\end{align*}
Therefore, we obtain
	\begin{equation}\label{tau}
	\|x(t)-x^{*}\| \leq \gamma \|\phi(t)-x^{*}\|_{s} {\rm e}^{-kt},
	\end{equation}
where
\begin{equation*}
\gamma=\sqrt{\lambda_{\min}(P)^{-1}\left[\lambda_{\max}(P)+\frac{1-{\rm e}^{-2k\tau_{1}}}{2k}\lambda_{\max}(Q)+2\max(d_{i})+\frac{2k\tau_{2}-1+{\rm e}^{-2k\tau_{2}}}{4k^{2}}\lambda_{\max}(H)\right]}>0.
\end{equation*}
Then, by  Definition \ref{exponentially} and \eqref{tau}, the equilibrium point of the NNMMD~\eqref{eq:mtdnn} is exponentially stable if \eqref{2LMI} holds.
\end{proof}

\begin{rem}\label{rem:spec}
When $\tau_2 = 0$, the NNMMD~\eqref{eq:mtdnn} reduces to  the  neural network model with discrete delay (NNMDD)
\begin{equation}\label{eq:dtdnn}
\left\{
\begin{aligned}
 &\frac{{\rm {d}}x(t)}{{\rm {d}}t}=-(A+2I)x(t)+ P_{\Omega}[2x(t)] + x(t - \tau_1) + b,\\
 &x(t) =\phi(t), ~ t \in [- \tau_1, 0],
\end{aligned}
\right.
\end{equation}
where $\tau_1 \geq 0$ denotes the transmission delay and $\phi(t) \in \mathcal{C}_{\tau_1}$.
\end{rem}

For the NNMDD~\eqref{eq:dtdnn}, we have the following results.

\begin{cor}
For each $\phi\in \mathcal{C}_{\tau_1}$, there exists a unique continuous solution to the  NNMDD~\eqref{eq:dtdnn} in the global time interval $[0, +\infty)$.
\end{cor}

\begin{cor}\label{the3}
If there exist two symmetric positve-definite matrices $\mathcal{P} \in \mathbb{R}^{n \times n}$ and $\mathcal{Q}\in \mathbb{R}^{n \times n}$, a positive-definite diagonal matrix $\mathcal{D} \in \mathbb{R}^{n \times n}$ and a constant $k > 0$ such that
		\begin{equation}\label{1LMI}
	\mathcal{R}=
	\left[
	\begin{array}{ccc}
	2k\mathcal{P}-A^{\top}\mathcal{P}-\mathcal{P}A-4\mathcal{P}+\mathcal{Q} & \mathcal{P}-\mathcal{D}A^{\top}+2k\mathcal{D} & \mathcal{P} \\
	\mathcal{P}-A\mathcal{D}+2k\mathcal{D} & -2\mathcal{D} & \mathcal{D} \\
	\mathcal{P} & \mathcal{D} & -{\rm e}^{-2k\tau_1}\mathcal{Q}
	\end{array}
	\right] \prec 0,
	\end{equation}
then the equilibrium point of the NNMDD~\eqref{eq:dtdnn} is exponentially stable.
\end{cor}
\begin{proof}
The proof is similar to that of Theorem \ref{stable:mtdnn} after defining the following Lyapunov-Krasovskii functional
\begin{align*}
V(x(t)) =& {\rm e}^{2kt} (x(t) - x^{*})^{\top} \mathcal{P} (x(t) - x^{*})+ 2{\rm e}^{2kt} \sum\limits_{i=1}^{n} d_{i} \int_{x_{i}^{*}}^{x_{i}(t)} (P_{\Omega}[2z_{i}] - P_{\Omega}[2x_{i}^{*}]) {\rm d}z_{i} \\\nonumber
	&\quad+ \int_{t - \tau_1}^{t} {\rm e}^{2ks} (x(s) - x^{*})^{\top} \mathcal{Q} (x(s) - x^{*}){\rm d}s.
\end{align*}
Hence, we omit the detail to save space.
\end{proof}

\section{Simulation results}\label{sec:exam}
In this section, three simulation examples are given to illustrate the effectiveness of the proposed models. All experiments are implemented in MATLAB R2021b with a machine precision $2.22\times 10^{-16}$ on a PC Windows $10$ operating system with an Intel i7-8750 CPU and 8 GB RAM. We use the solver ``DDESD'' to solve delay differential equations \cite{S2005} and ``ODE23'' is used for solving ordinary differential equations.

\begin{example}\label{ex:example1}
Consider the AVE \eqref{eq:ave} with
\begin{equation*}
A=\left[\begin{array}{ccc}
4 & 0 & 0 \\
2 & 2  & 8 \\
0  &0 & 4
\end{array}
\right] \quad \text{and}\quad
b=\left[
\begin{array}{c}
	0\\
	1\\
    0
\end{array}
\right].
\end{equation*}
It is clear that this AVE has a unique solution $x^{*}=[0,1,0]^{\top}$ while $\|A^{-1}\|\approx 1.1677 > 1$.

For the NNMMD~\eqref{eq:mtdnn}, we first let $\tau_{1}=1$ and $\tau_{2}=0.5$. For $k=0.01$, by using the appropriate linear matrix inequality (LMI) solver \cite{pg1994} to get the feasible numerical solution of the LMI \eqref{2LMI}, we have
\begin{equation*}
\begin{array}{ll}
P=\left[\begin{array}{ccc}
	4.9653 & -0.2704& 1.0868\\
-0.2704 &  2.0143 &-1.0943\\
 1.0868  & -1.0943 &  9.6033
\end{array}
\right]\succ 0,& Q=\left[\begin{array}{ccc}
	15.5985 & 0.1188  & 0.1670\\
	0.1188& 7.2381 & 0.3106\\
 0.1670 &  0.3106 & 16.5041
\end{array}
\right]\succ 0,\\
H=\left[\begin{array}{ccc}
	9.2818 & -0.0961 & 0.2513\\
	-0.0961  & 6.8007 & -0.4378\\
0.2513 & -0.4378 &  10.3333
\end{array}
\right]\succ 0,&
D=\left[\begin{array}{ccc}
		 2.7522   & 0 & 0\\
		0 & 0.8604     &0\\
	0& 0 & 3.7904
	\end{array}
\right]\succ 0,\\
T_{1}=\left[\begin{array}{ccc}
0.7772  & -0.0664 & 0.1755\\
-0.0508 & 0.3602 & -0.2314\\
0.1481 & -0.2432 &  1.4252
\end{array}
\right],&
T_{2}=\left[\begin{array}{ccc}
-1.3824  & -0.0044 & 0.0846\\
-0.1382 &  -0.3564 & -0.5537\\
0.1792   & -0.0928  & -1.2108
\end{array}
\right].
\end{array}
\end{equation*}
By Theorem~\ref{stable:mtdnn}, the equilibrium point of the NNMMD~\eqref{eq:mtdnn} is exponentially stable. For the NNMDD~\eqref{eq:dtdnn}, we first let $\tau_{1}=0.01$. Then for $k=0.01$, the feasible numerical solution of the LMI~\eqref{1LMI} could be as
\begin{align*}
P&=
\left[
\begin{array}{ccc}
0.1862  & -0.0187  &  0.0484\\
-0.0187  &  0.1924  & -0.1076\\
0.0484 &  -0.1076  &  0.3300
\end{array}
\right]\succ 0,\quad Q=
\left[
\begin{array}{ccc}
0.8601 &   0.0039  & -0.0278\\
0.0039 &   0.6984  &  0.0540\\
-0.0278  &  0.0540  &  0.7507
\end{array}
\right]\succ 0,\\
D&=
\left[
\begin{array}{ccc}
0.1199   &      0  &       0\\
0   & 0.3348   &      0\\\nonumber
0      &   0  &  0.0582
\end{array}
\right]\succ 0.
\end{align*}
By Corollary~\ref{the3}, the equilibrium point of the NNMDD~\eqref{eq:dtdnn} is exponentially stable. Figures \ref{F1}-\ref{F2} show the corresponding transient behaviors of the NNMMD~\eqref{eq:mtdnn} and the NNMDD~\eqref{eq:dtdnn} with the initial function $[2\cos(t),2\cos(t),-5\cos(t)]^{\top}$, from which we can find that they converge to the equilibrium point $x^{*}=[0,1,0]^{\top}$, the unique solution of the AVE~\eqref{eq:ave} in this example. Since $\|A^{-1}\|>1$, the existing neural network models proposed in \cite{cyyh2021,jyfc2023} are not available in theory. However, numerically, Figures \ref{FIX}-\ref{CYYH} illustrate the transient behaviors of CYYH \cite{cyyh2021} with $\gamma =2$, FIX \cite{jyfc2023} with $\xi_{1}=0.5, \xi_{2}=1.5, \tau_{1}=2, \tau_{2}=2$, which show that they numerically converge to the unique solution of the AVE~\eqref{eq:ave}.

\begin{figure}[htbp]
	{
		\begin{minipage}{.5\linewidth}
			
			\includegraphics[scale=0.5]{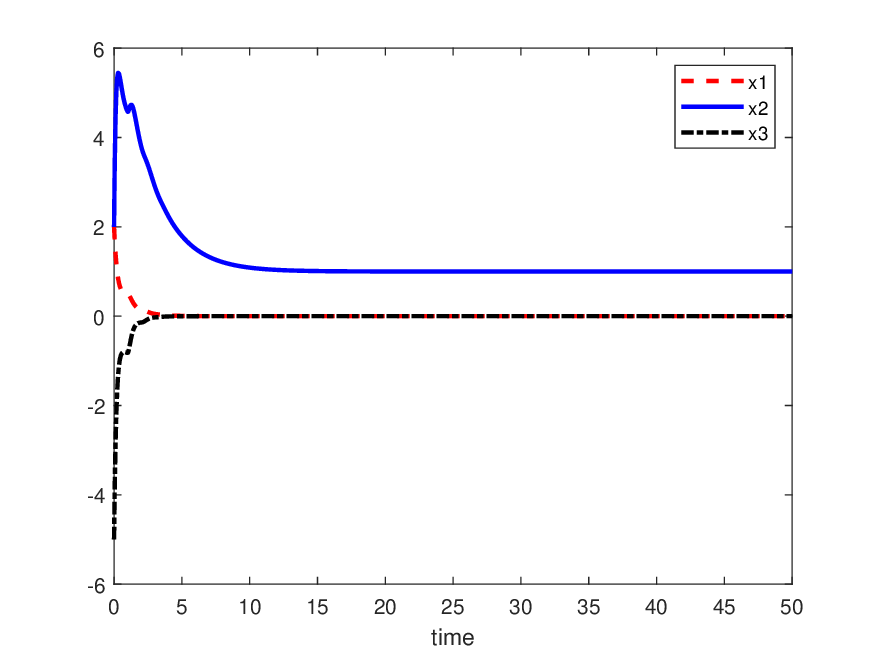}
			\caption{The transient behaviors of NNMMD~\eqref{eq:mtdnn}.}\label{F1}
		\end{minipage}
	}
\quad
	{
		\begin{minipage}{.5\linewidth}
			\centering
			\includegraphics[scale=0.5]{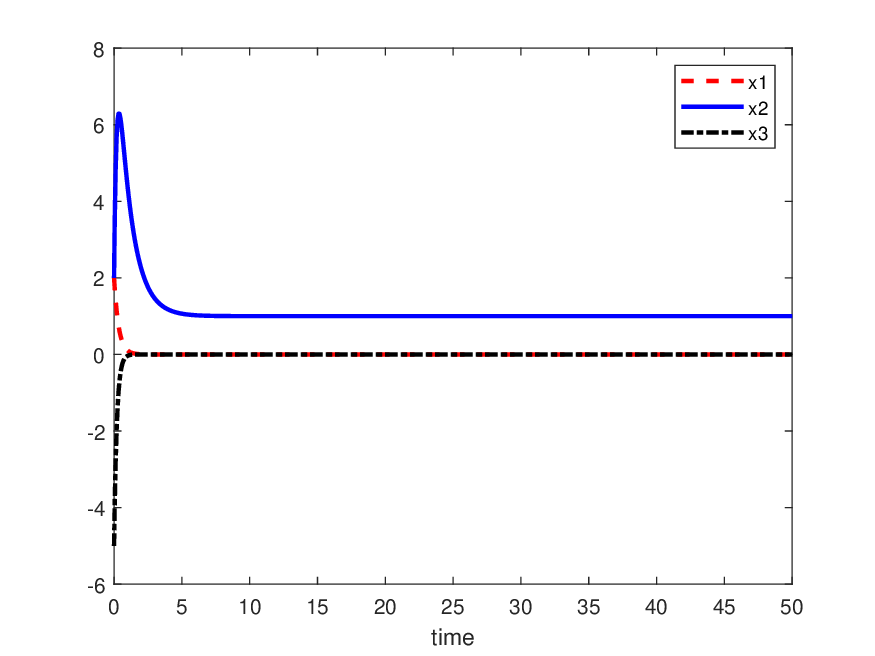}
			\caption{The transient behaviors of NNMDD~\eqref{eq:dtdnn}.}\label{F2}
		\end{minipage}
	}
	\quad
	{
		\begin{minipage}{.5\linewidth}
			\centering
			\includegraphics[scale=0.5]{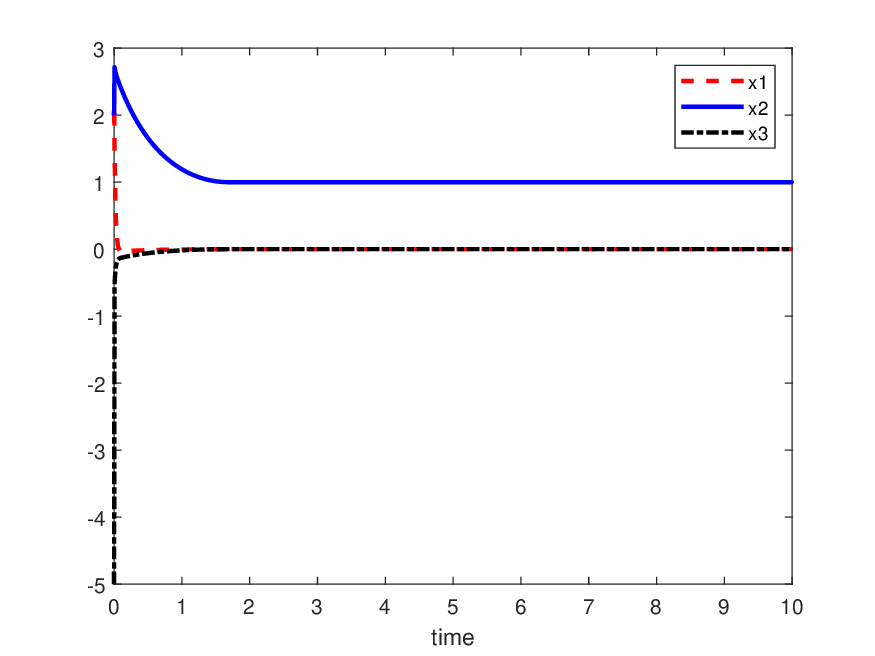}
			\caption{The transient behaviors of FIX. The initial vector is $[2, 2, -5]^{\top}$.}\label{FIX}
		\end{minipage}
	}
\quad
{
	\begin{minipage}{.5\linewidth}
		\centering
		\includegraphics[scale=0.5]{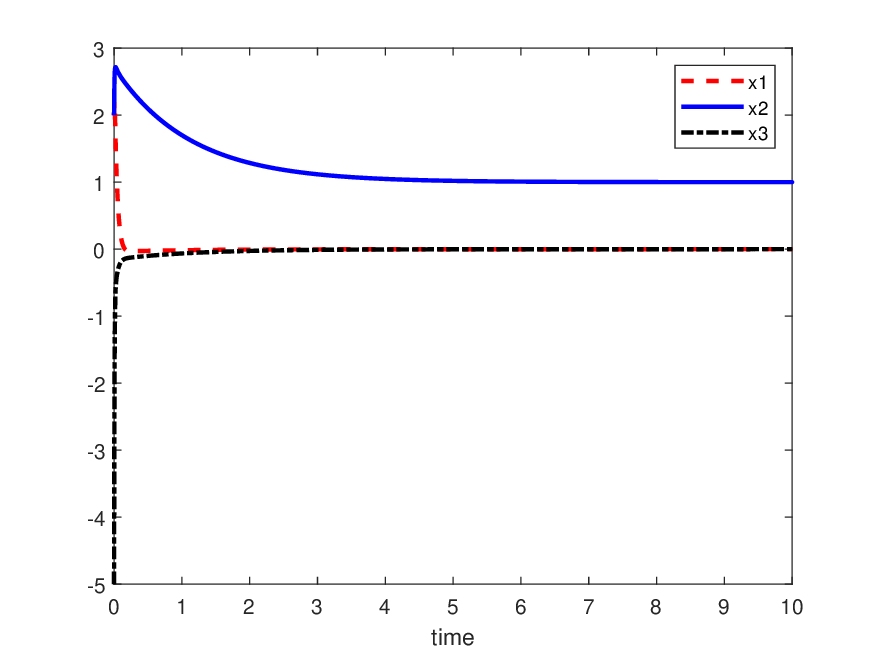}
		\caption{The transient behaviors of CYYH. The initial vector is $[2, 2, -5]^{\top}$}\label{CYYH}
	\end{minipage}
}
\end{figure}
\end{example}

In the following, we will give a example that CYYH and FIX are not available both in theory and in numeric while  NNMMD~\eqref{eq:mtdnn} and NNMDD~\eqref{eq:dtdnn} are still guaranteed.

\begin{example}\label{ex:example2}
Consider the AVE \eqref{eq:ave} with
\begin{equation*}
	A=
	\left[
	\begin{array}{ccc}
	2 & 0 & 0 \\
	0 & 0.9  & 0 \\
	0  &0 & 2
	\end{array}
	\right] ,
	~	b=
	\left[
	\begin{array}{c}
	0\\
	0\\
	-1
	\end{array}
	\right].
	\end{equation*}
It is clear that this AVE has the unique solution $x^{*}=[0,0,-0.33]^{\top}$ and $\|A^{-1}\|\approx 1.1111 > 1$. For the NNMMD~\eqref{eq:mtdnn}, we set $\tau_{1}=1$ and  $\tau_{2}=0.5$. For $k=0.01$, the feasible numerical solution of the LMI \eqref{2LMI} can be
\begin{equation*}
	\begin{array}{ll}
		P=\left[\begin{array}{ccc}
			0.3272 & 0& 0\\
			0& 0.4303  &0\\
			0  & 0 & 0.3272
		\end{array}
		\right]\succ 0,& Q=\left[\begin{array}{ccc}
		1.1880   & 0 & 0\\
			0& 1.1660  & 0\\
			0 &  0 & 1.1880
		\end{array}
		\right]\succ 0,\\
		H=\left[\begin{array}{ccc}
			0.7349& 0 & 0\\
			0  & 0.7511  & 0\\
			0 & 0 & 0.7349
		\end{array}
		\right]\succ 0,&
		D=\left[\begin{array}{ccc}
			0.4773   & 0 & 0\\
			0 & 0.5482    &0\\
			0& 0 & 0.4773
		\end{array}
		\right]\succ 0,\\
		T_{1}=\left[\begin{array}{ccc}
			0.0483 & 0 & 0\\
			0&  0.0845  & 0\\
			0 & 0 & 0.0483
		\end{array}
		\right],&
		T_{2}=\left[\begin{array}{ccc}
			-0.1706 & 0 & 0\\
			0& -0.0808  & 0\\
			0  & 0 & -0.1706
		\end{array}
		\right].
	\end{array}
\end{equation*}
For the NNMDD~\eqref{eq:dtdnn}, let $\tau_{1}=0.01$. Then for $k=0.01$, the feasible numerical solution of the LMI \eqref{1LMI} is
\begin{align*}
P&=
\left[
\begin{array}{ccc}
4.6750  & 0 &   0\\
0  & 6.0414  &  0\\
0  &  0 &  4.6750
\end{array}
\right]\succ 0,\quad Q=
\left[
\begin{array}{ccc}
18.2485  & 0 &  0\\
0  &  17.4525   & 0\\
0 & 0  &  18.2485
\end{array}
\right]\succ 0,\\
D&=
\left[
\begin{array}{ccc}
 4.7718     &    0    &     0\\
0   & 7.3621      &    0\\\nonumber
0     &    0  &  4.7718
\end{array}
\right]\succ 0.
\end{align*}
Figure \ref{F7} shows the transient behaviors of $x(t)$ in NNMMD~\eqref{eq:mtdnn} and  NNMDD~\eqref{eq:dtdnn} with initial function $[2\cos(t),-2\cos(t),-5\cos(t)]^\top$, and the  transient behaviors of CYYH ($\gamma =2$) and FIX ($\xi_{1}=0.5, \xi_{2}=1.5, \tau_{1}=2$ and  $\tau_{2}=2$) with initial vector $[1,1,2]^{\top}$. It follows from \ref{F7} that NNMMD~\eqref{eq:mtdnn} and  NNMDD~\eqref{eq:dtdnn} converge to the unique solution of the AVE while CYYH and FIX fail.

\begin{figure}
	\subfigure{
		\begin{minipage}{.5\linewidth}
			\centering
			\includegraphics[scale=0.5]{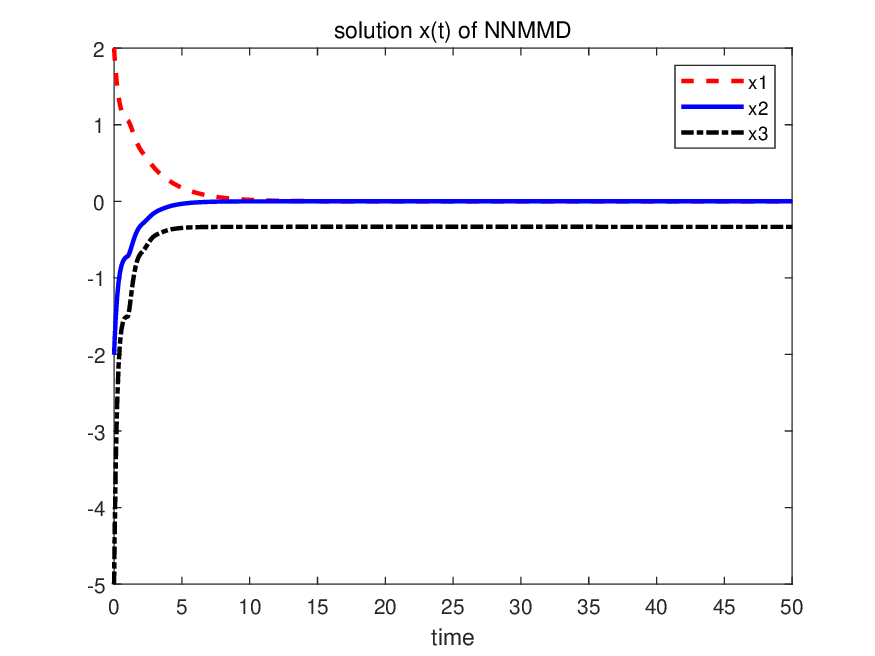}
		\end{minipage}
	}
	\quad
	\subfigure	{
		\begin{minipage}{.5\linewidth}
			\centering
			\includegraphics[scale=0.5]{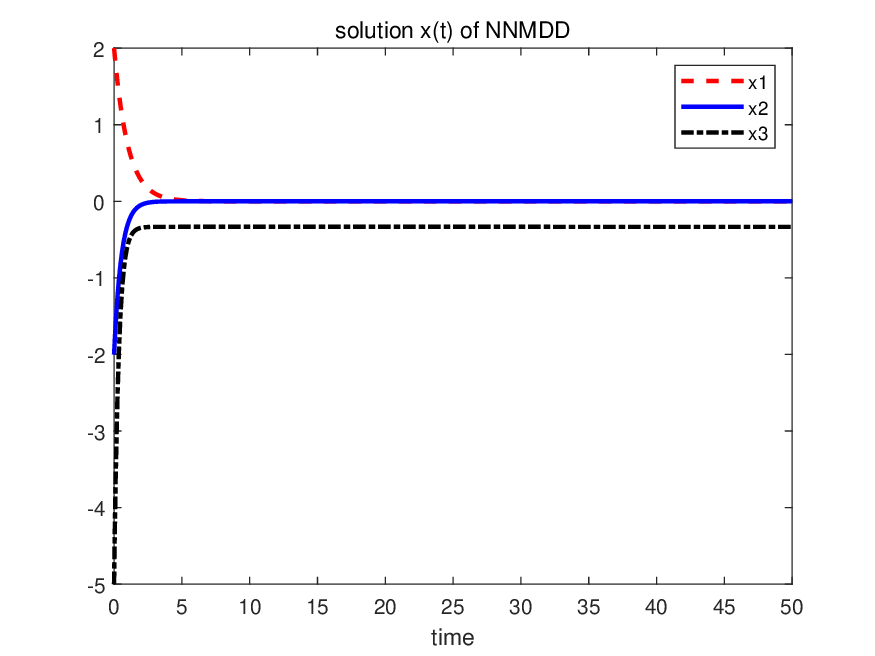}
		\end{minipage}
	}
	\quad
	\subfigure	{
		\begin{minipage}{.5\linewidth}
			\centering
			\includegraphics[scale=0.5]{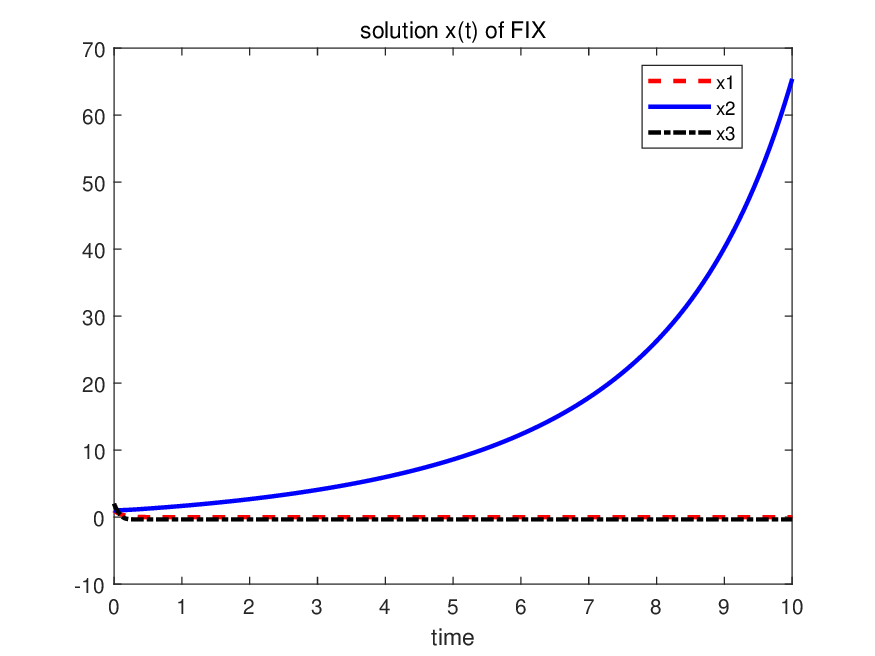}
			
		\end{minipage}
	}
	\quad
	\subfigure{
		\begin{minipage}{.5\linewidth}
			\centering
			\includegraphics[scale=0.5]{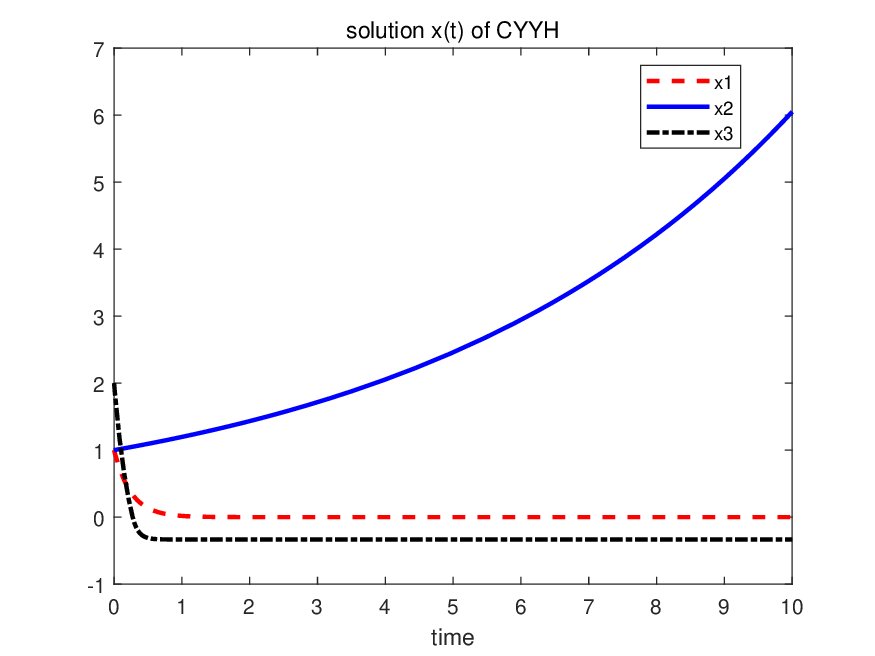}
			
		\end{minipage}
	}
	\caption{The transient behaviors of NNMMD~\eqref{eq:mtdnn}, NNMDD~\eqref{eq:dtdnn}, FIX, CYYH.}\label{F7}
\end{figure}
\end{example}

The following example shows that NNMMD~\eqref{eq:mtdnn} and NNMDD~\eqref{eq:dtdnn} are also exponentially stable when $1$ is an eigenvalue of $A$. As mentioned in \cite{cyyh2021}, the methods proposed in \cite{maee2017,maer2018} are not available in this case. In addition, in the case that the AVE \eqref{eq:ave} has multiple solutions, under the same initial condition,  NNMMD~\eqref{eq:mtdnn} and NNMDD~\eqref{eq:dtdnn} can get different solutions of the AVE \eqref{eq:ave} by changing the delay parameters.

\begin{example}\label{ex:example3}
Considering the AVE \eqref{eq:ave} with
	\begin{equation*}
	A=
	\left[
	\begin{array}{ccc}
	1 & 0 & 0 \\
	0 & 1  & 0 \\
	0  &0 & 1
	\end{array}
	\right] ,
~	b=
	\left[
	\begin{array}{c}
	0\\
	0\\
	-1
	\end{array}
	\right].
	\end{equation*}
Obviously, the AVE has infinitely many solutions and thus both NNMMD~\eqref{eq:mtdnn} and NNMDD~\eqref{eq:dtdnn} have infinitely many equilibrium points. As shown in Figures \ref{F8}-\ref{F9}, both NNMMD~\eqref{eq:mtdnn} and NNMDD~\eqref{eq:dtdnn} can obtain different solutions of AVE \eqref{eq:ave} with different delays while the initial function $[\cos(t),2\cos(t),-2\cos(t)]^{\top}$ is frozen. Theoretically, the corresponding LMI is solvable. For example, for the NNMMD~\eqref{eq:mtdnn}, let $\tau_{1}=5$ and $\tau_{2}=0.5$, then for  $k=0.01$,  the LMI \eqref{2LMI} has the solution
\begin{equation*}
	\begin{array}{ll}
		P=\left[\begin{array}{ccc}
		0.4174 & 0& 0\\
			0& 0.4174  &0\\
			0  & 0 & 0.4174
		\end{array}
		\right]\succ 0,& Q=\left[\begin{array}{ccc}
			1.3234    & 0 & 0\\
			0& 1.3234   & 0\\
			0 &  0 & 1.3234
		\end{array}
		\right]\succ 0,\\
		H=\left[\begin{array}{ccc}
		0.7985& 0 & 0\\
			0  & 0.7985  & 0\\
			0 & 0 & 0.7985
		\end{array}
		\right]\succ 0,&
		D=\left[\begin{array}{ccc}
		0.5021   & 0 & 0\\
			0 & 0.5021    &0\\
			0& 0 & 0.5021
		\end{array}
		\right]\succ 0,\\
		T_{1}=\left[\begin{array}{ccc}
			0.0698 & 0 & 0\\
			0& 0.0698   & 0\\
			0 & 0 & 0.0698
		\end{array}
		\right],&
		T_{2}=\left[\begin{array}{ccc}
			-0.1400   & 0 & 0\\
			0& -0.1400    & 0\\
			0  & 0 & -0.1400
		\end{array}
		\right].
	\end{array}
\end{equation*}
For the NNMDD~\eqref{eq:dtdnn}, let $\tau_{1}=0.01$, then for  $k=0.01$,  the LMI \eqref{1LMI} has the solution
\begin{align*}
P&=
\left[
\begin{array}{ccc}
6.4771  & 0 &   0\\
0  & 6.4771  &  0\\
0  &  0 &  6.4771
\end{array}
\right]\succ 0,
\quad Q=
\left[
\begin{array}{ccc}
19.2798  & 0 &  0\\
0  &  19.2798  & 0\\
0 & 0  &  19.2798
\end{array}
\right]\succ 0,\\
D&=
\left[
\begin{array}{ccc}
7.8471     &    0    &     0\\
0   & 7.8471     &    0\\\nonumber
0     &    0  &  7.8471
\end{array}
\right].
\end{align*}

\begin{figure}
	\subfigure{
		\begin{minipage}{.5\linewidth}
			\centering
			\includegraphics[scale=0.5]{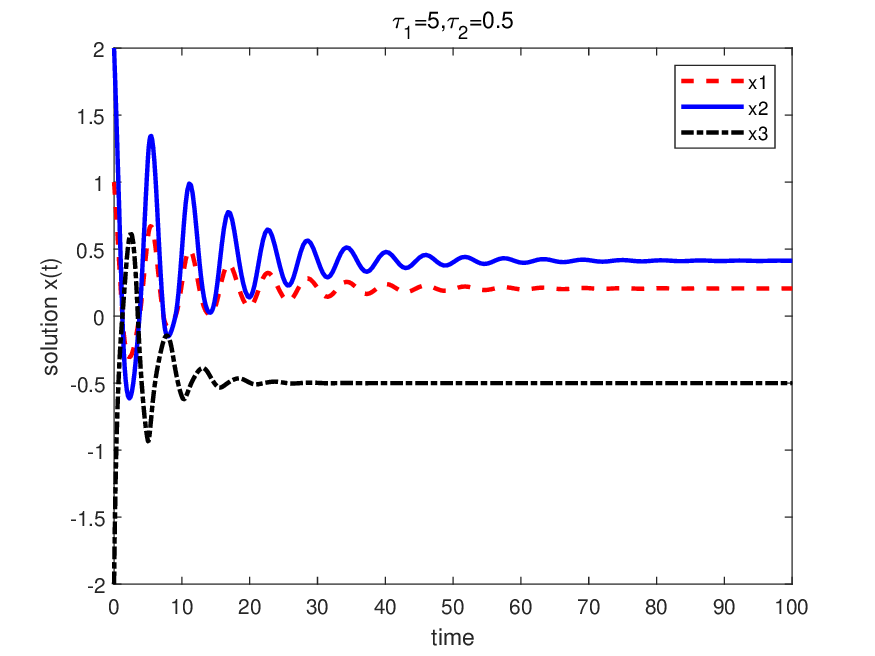}
		\end{minipage}
	}
\quad
\subfigure	{
	\begin{minipage}{.5\linewidth}
		\centering
		\includegraphics[scale=0.5]{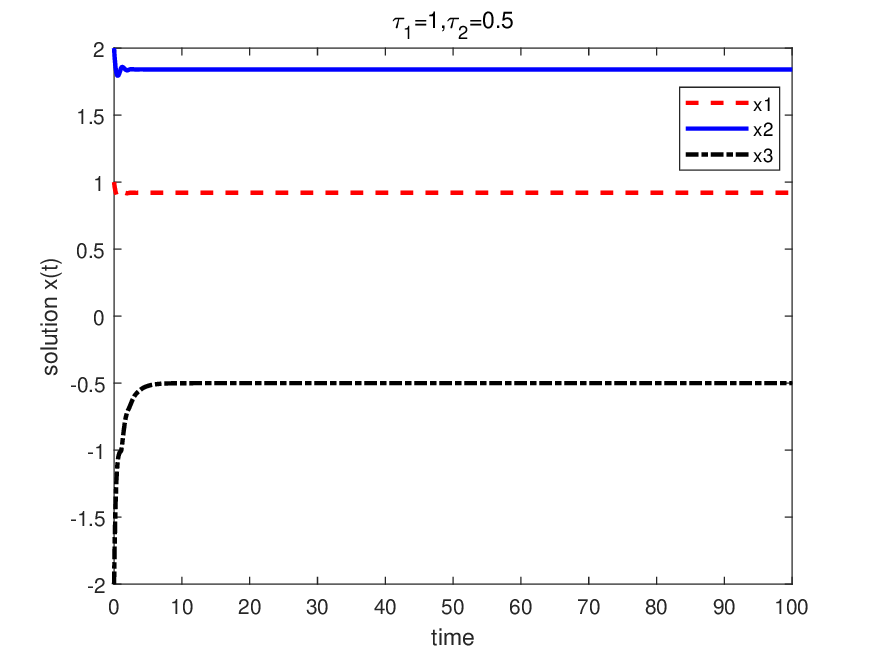}
	\end{minipage}
}
\quad
\subfigure	{
	\begin{minipage}{.5\linewidth}
		\centering
		\includegraphics[scale=0.5]{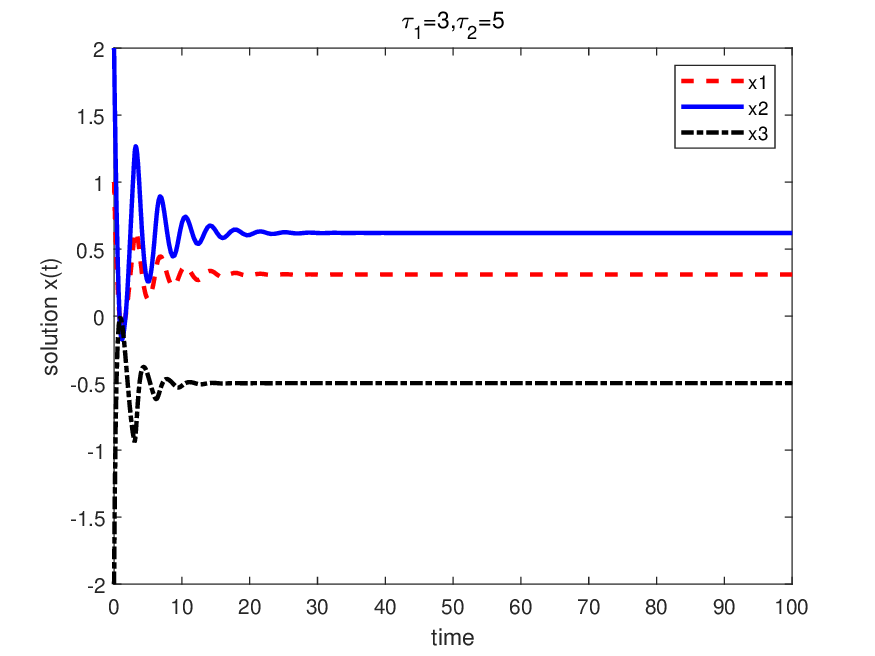}
	\end{minipage}
}
\quad
	\subfigure{
	\begin{minipage}{.5\linewidth}
		\centering
		\includegraphics[scale=0.5]{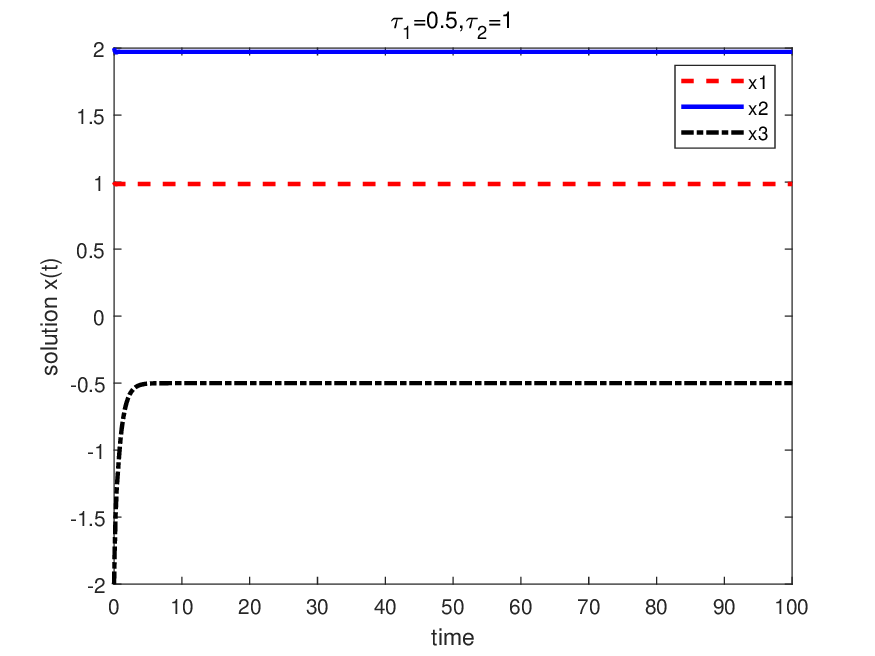}
	\end{minipage}
}
\caption{The transient behaviors of the NNMMD~\eqref{eq:mtdnn} with different values of $\tau_{1}$ and $ \tau_{2}$. }\label{F8}
\end{figure}

\begin{figure}
\subfigure	{
	\begin{minipage}{.5\linewidth}
		\centering
		\includegraphics[scale=0.5]{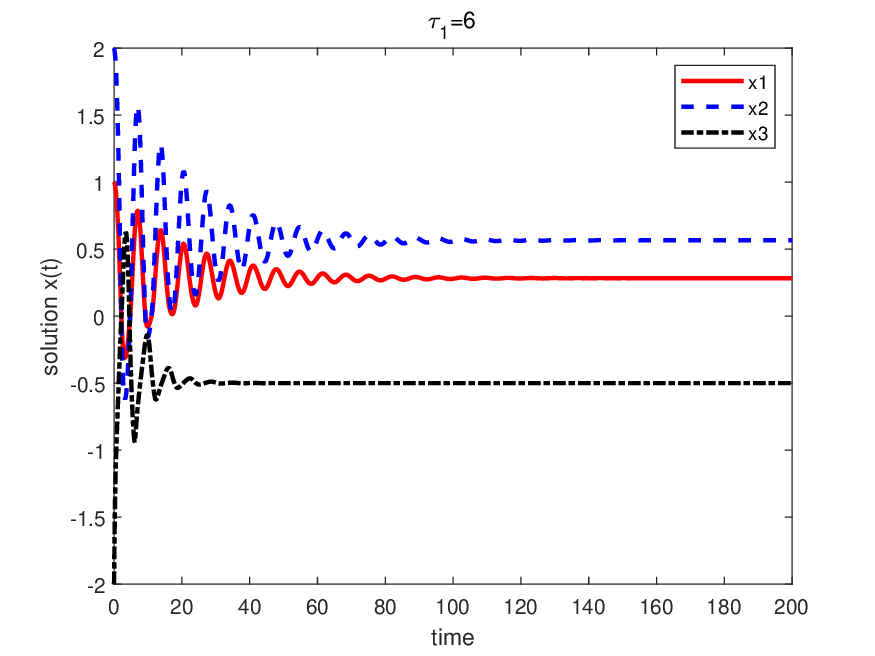}
		
	\end{minipage}
}
\quad
	\subfigure	{
		\begin{minipage}{.5\linewidth}
			\centering
			\includegraphics[scale=0.5]{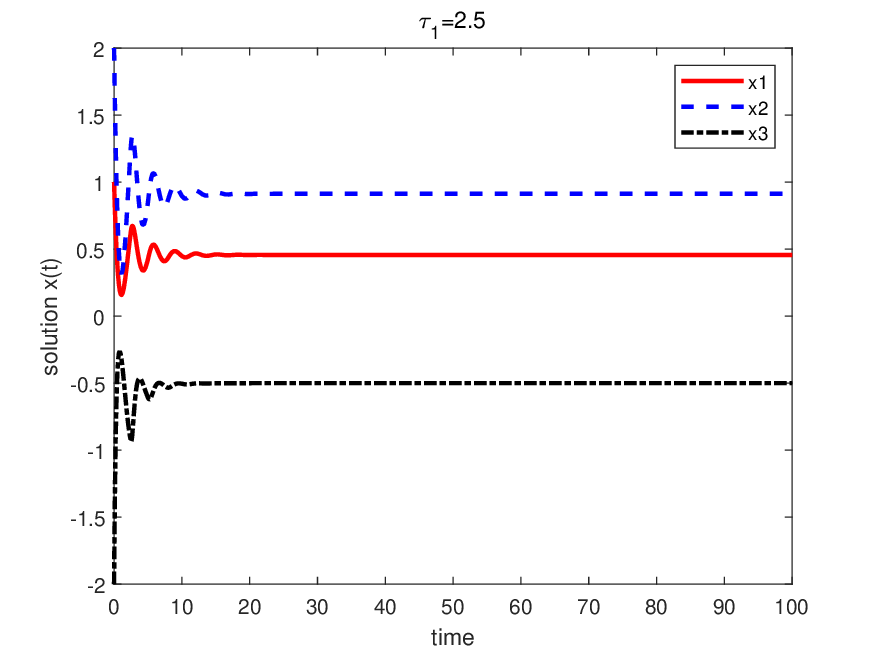}
			
		\end{minipage}
	}
\quad
	\subfigure{
		\begin{minipage}{.5\linewidth}
			\centering
			\includegraphics[scale=0.5]{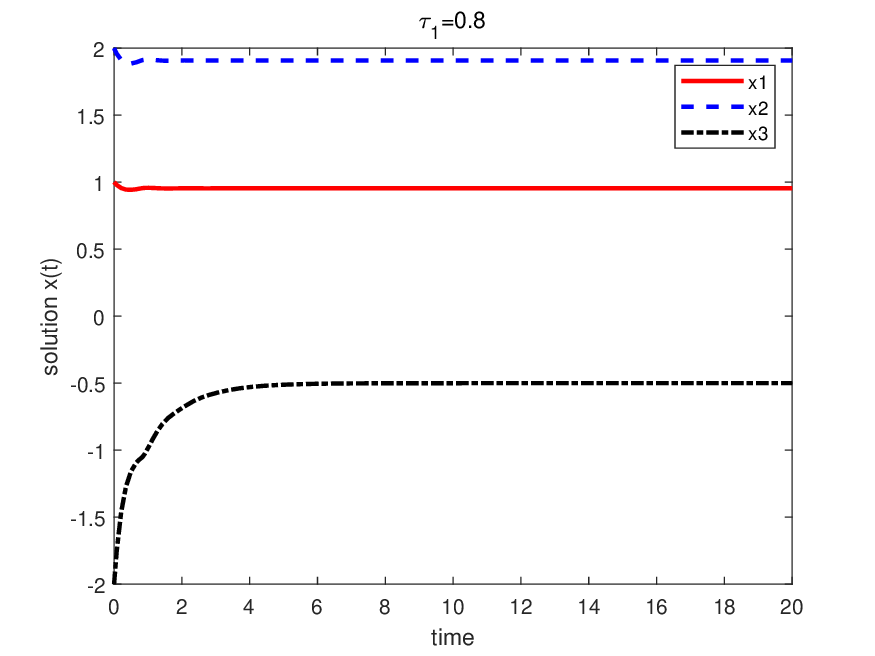}
			
		\end{minipage}
	}
\quad
\subfigure{
	\begin{minipage}{.5\linewidth}
		\centering
		\includegraphics[scale=0.5]{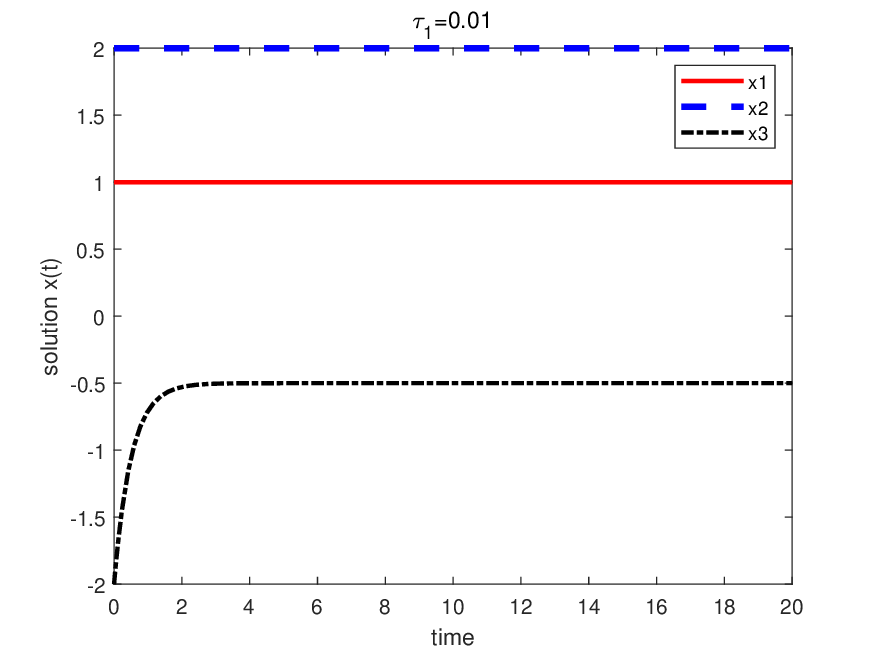}
		
	\end{minipage}
}
\caption{The transient behaviors of the NNMDD~\eqref{eq:dtdnn} with different values of  $\tau_{1}$. }\label{F9}
\end{figure}
\end{example}

\begin{example}\label{ex:example4}
	Consider the AVE \eqref{eq:ave} with
	\begin{equation*}
		A=\left[\begin{array}{ccc}
			8 & -1 &0 \\
			-1 & 8  & -1 \\
			0  &-1 & 8
		\end{array}
		\right] \quad \text{and}\quad
		b=\left[
		\begin{array}{c}
			-10\\
			9\\
			-10
		\end{array}
		\right].
	\end{equation*}
	This AVE has a unique solution $x^{*}=[-1,1,-1]^{\top}$ with $\|A^{-1}\|< 1$.
	
	For the NNMMD~\eqref{eq:mtdnn}, we let $\tau_{1}=1$, $\tau_{2}=0.5$. For $k=0.01$, by using the appropriate linear matrix inequality (LMI) solver to get the feasible numerical solution of the LMI \eqref{2LMI}, we have
	\begin{equation*}
		\begin{array}{ll}
			P=\left[\begin{array}{ccc}	
			53.9876 &   0.9903  &  0.2375\\
			0.9903  & 53.8978  &  0.9903\\
			0.2375  &  0.9903  & 53.9876
			\end{array}
			\right]\succ 0,& Q=\left[\begin{array}{ccc}		
				263.4657  &  2.1503 &  -0.8182\\
				2.1503 & 262.6324  &  2.1503\\
				-0.8182 &   2.1503 & 263.4657
			\end{array}
			\right]\succ 0,\\
			H=\left[\begin{array}{ccc}
				156.5716 &   0.7261 &   0.0486\\
				0.7261 & 156.5920  &  0.7261\\
				0.0486  &  0.7261 & 156.5716
			\end{array}
			\right]\succ 0,&
			D=\left[\begin{array}{ccc}
			19.7867   &      0     &    0\\
			0  & 19.4132    &     0\\
			0     &    0  & 19.7867
			\end{array}
			\right]\succ 0,\\
			T_{1}=\left[\begin{array}{ccc}
			6.3425  &  0.5075  &  0.0731\\
			0.5107  &  6.3758  &  0.5107\\
			0.0731  &  0.5075  &  6.3425
			\end{array}
			\right],&
			T_{2}=\left[\begin{array}{ccc}
				-15.4606  &  0.4365 &   0.0481\\
				0.4133 & -15.1657   & 0.4133\\
				0.0481  &  0.4365 & -15.4606
			\end{array}
			\right].
		\end{array}
	\end{equation*}

\begin{figure}[htbp]
	{
		\begin{minipage}{.5\linewidth}
			
			\includegraphics[scale=0.5]{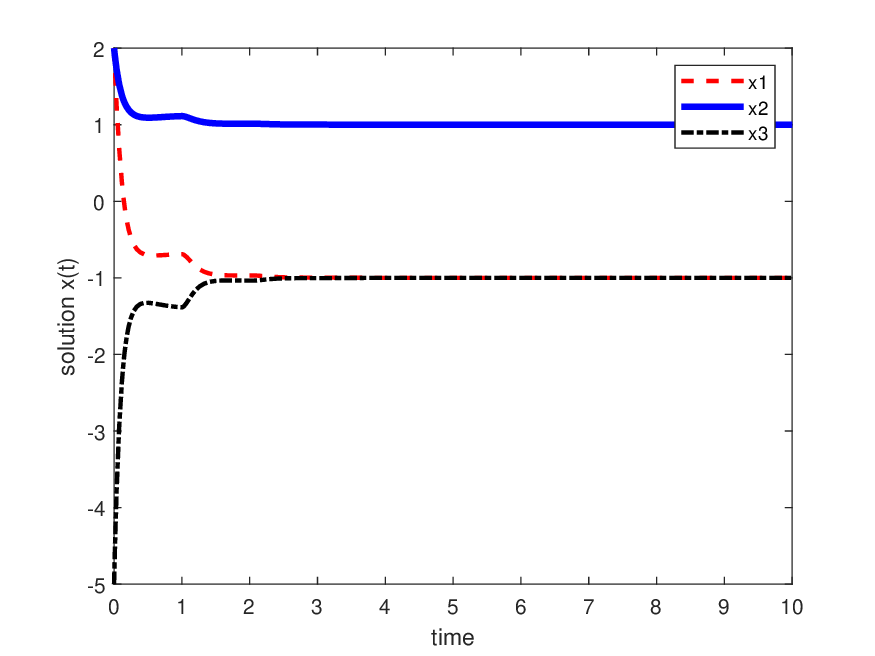}
			\caption{The transient behaviors of NNMMD~\eqref{eq:mtdnn}.}\label{4F1}
		\end{minipage}
	}
	\quad
	{
		\begin{minipage}{.5\linewidth}
			\centering
			\includegraphics[scale=0.5]{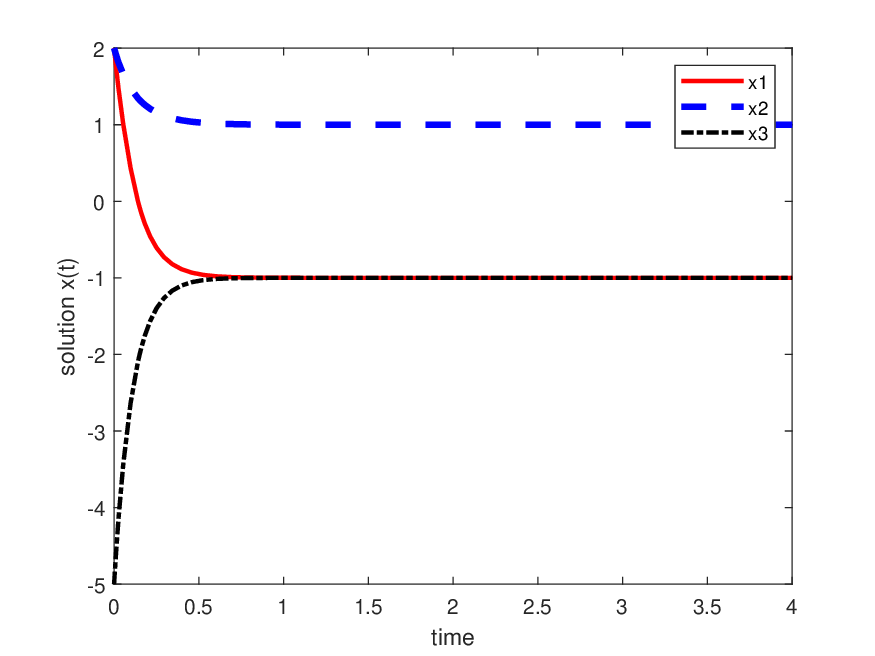}
			\caption{The transient behaviors of NNMDD~\eqref{eq:dtdnn}.}\label{4F2}
		\end{minipage}
	}
\end{figure}
	By Theorem~\ref{stable:mtdnn}, the equilibrium point of the NNMMD~\eqref{eq:mtdnn} is exponentially stable. For the NNMDD~\eqref{eq:dtdnn}, we first let $\tau_{1}=0.01$. Then for $k=0.01$, the feasible numerical solution of the LMI~\eqref{1LMI} could be as
	\begin{align*}
		P&=
		\left[
		\begin{array}{ccc}
			0.1138  &  0.0114  &  0.0011\\
			0.0114 &   0.1149  &  0.0114\\
			0.0011  &  0.0114   & 0.1138
		\end{array}
		\right]\succ 0,\quad Q=
		\left[
		\begin{array}{ccc}
			1.1119  & -0.0005 &  -0.0002\\
			-0.0005  &  1.1115 &  -0.0005\\
			-0.0002  & -0.0005  &  1.1119
		\end{array}
		\right]\succ 0,\\
		D&=
		\left[
		\begin{array}{ccc}
			 0.0370     &    0   &      0\\
			0  &  0.0364     &    0\\
			0     &    0  &  0.0370
		\end{array}
		\right]\succ 0.
	\end{align*}
	By virtue of Corollary~\ref{the3}, the equilibrium point of the NNMDD~\eqref{eq:dtdnn} is exponentially stable. Figures \ref{4F1}-\ref{4F2} show the corresponding transient behaviors of the NNMMD~\eqref{eq:mtdnn} and the NNMDD~\eqref{eq:dtdnn} with the initial function $[2\cos(t),2\cos(t),-5\cos(t)]^{\top}$. It can be found that the corresponding transient behaviors of the NNMMD~\eqref{eq:mtdnn} and the NNMDD~\eqref{eq:dtdnn} converge to the equilibrium point $x^{*}=[-1,1,-1]^{\top}$, the unique solution of the AVE~\eqref{eq:ave} in this example.

\end{example}

\section{Conclusion}\label{sec:conc}
Two neural network models with delays are proposed to solve the AVE \eqref{eq:ave}. The exponential stabilities of the proposed models have been proved. Specially, the NNMMD~\eqref{eq:mtdnn} is exponentially stable if \eqref{2LMI} holds and the NNMDD~\eqref{eq:dtdnn} is exponentially stable if \eqref{1LMI} holds. Moreover, theoretically, both \eqref{2LMI} and \eqref{1LMI} can be satisfied when $\|A^{-1}\|\le 1$ or $\|A^{-1}\|>1$. Numerical experiments demonstrate the effectiveness of the proposed models.

\end{document}